\documentclass[a4paper,12pt,oneside]{article}
\usepackage{amsfonts}
\usepackage[a4paper,hmargin={2.54cm,2.54cm},vmargin={3.17cm,3.17cm}]{geometry}
\usepackage{amsmath,amssymb,amsthm,amsfonts}
\usepackage{bbm}
\usepackage{stmaryrd}
\usepackage{amsthm}
\numberwithin{equation}{section}

\usepackage{abstract}
\usepackage{graphicx}
\newtheoremstyle{mystyle}{12pt}{12pt}{\itshape}{0cm}{\bfseries}{}{1em}{}
\theoremstyle{mystyle}
\newtheorem{definition}{Definition}[section]
\newtheorem{theorem}{Theorem}[section]

\newtheorem{lemma}{Lemma}[section]
\newtheorem{proposition}{Proposition}[section]
\newtheorem{corollary}{Corollary}[section]

\newtheorem{remark}{Remark}[section]

\begin{document}

\title{\bf{Asymptotic Joint Distribution of Extreme Eigenvalues of the Sample Covariance Matrix in the Spiked Population Model}}
 \author{Dai Shi\footnote{Institute of Computational and Mathematical Engineering, Stanford University, Stanford, 94305}}

\date{}

\renewcommand{\thepage}{\roman{page}}
\setcounter{page}{1} 
\maketitle\renewcommand{\thepage}{\arabic{page}}\setcounter{page}{1}

\begin{abstract}
In this paper, we consider a data matrix $X\in\mathbb{C}^{N\times M}$ where all the columns are i.i.d. samples being $N$ dimensional complex Gaussian of mean zero and covariance $\Sigma\in\mathbb{C}^{N\times N}$. Here the population matrix $\Sigma$ is of finite rank perturbation of the identity matrix. This is the ``spiked population model'' first proposed by Johnstone in \cite{21}. As $M, N\to\infty$ but $M/N = \gamma\in(1, \infty)$, we first establish in this paper the asymptotic distribution of the smallest eigenvalue of the sample covariance matrix $S:= XX^*/M$. It also exhibits a phase transition phenomenon proposed in \cite{1} --- the local fluctuation will be the generalized Tracy-Widom or the generalized Gaussian to be defined in the paper. 
Moreover we prove that the largest and the smallest eigenvalue are asymptotically independent as $M, N\to\infty$. 
\end{abstract}

\section{Introduction}\label{sec1}
Suppose we have $M$ independently  and identically distributed samples $x_1, \ldots, x_M\in\mathbb{C}^{N}$, each of which is an $N$ dimensional complex-valued vector. Here $M$ is the sample size and $N$ is the dimension of our data. We can then form the data matrix $X = (x_1, \ldots, x_M) \in\mathbb{C}^{N\times M}$ and further define its sample covariance matrix
\[
S := \frac{1}{M}XX^*\in\mathbb{C}^{N\times N}. 
\]
In this paper we are interested in the asymptotic joint distribution of the largest and the smallest eigenvalues of the matrix $S$. Below is the assumptions of our model.
\begin{itemize}
\item
All the data vectors $x_i$ are independently and identically distributed following the complex Gaussian distribution of mean zero and covariance matrix $\Sigma\in\mathbb{C}^{N\times N}$. Here $\Sigma$ is a non-random positive definite matrix. 
\item
$M, N\to\infty$ but their ratio $M/N := \gamma^2$ is a fixed amount in the interval $(1, \infty)$. 
\item
We denote $\ell_1\geq\ell_2\geq\ldots\geq\ell_N$ to be the eigenvalues of the matrix $\Sigma$. Then we assume that all of the $\ell_i$'s are equal to one except for only finite of them. That is, there exist fixed integers $r_1, r_2$ which are independent of $M, N$ such that 
\[
\ell_N \leq  \ell_{N-1}\leq\ldots \leq \ell_{N-r_1+1} < 1 < \ell_{r_2} \leq \ell_{r_2-1}\leq \ldots\leq \ell_1
\]
and
\[
\ell_{N-r_1} = \ell_{N-r_1-1} = \ldots = \ell_{r_2+2} = \ell_{r_2+1} = 1.
\]
\end{itemize}

The model defined above is the (complex) ``spiked population model'' proposed in \cite{21}. The unit eigenvalues represent pure noise, while the spiked eigenvalues represent true information. In real applications, we will encounter such models quite often. In mathematical imaging (see \cite{a1}), the observed spectrum of the sample covariance matrix indeed has some detached eigenvalues, representing the possible scatterers in the region. As another example, in mathematical finance (see \cite{a2}), each column of our data matrix represents the correlated return of each stock. The sample correlation matrix has some spiked large eigenvalues, representing the main factors driving the market, and some small eigenvalues, representing the linear dependence of these factors. Other possible applications include, but not restricted to, speech recognition (see \cite{a3}), physics mixture (see \cite{a4}) and statistical learning (see \cite{a5}). 

We define $(\lambda_1, \ldots, \lambda_N)$ as the eigenvalues of the sample covariance matrix $S$, where $\lambda_1\geq \lambda_2 \geq \ldots\geq\lambda_N$. In the null case where $\Sigma = I$, a lot of properties are known.  The empirical measure of $\{\lambda_i\}_{i=1}^N$, denoted by $F_N := \sum_{i=1}^N\delta(\lambda_i)/N$, will almost surely converge in distribution to the Mar\v{c}enko-Pastur law (see \cite{13}), whose density is defined by 
\begin{equation}\label{MP law}
F(\lambda) := \frac{\gamma^2}{2\pi \lambda}\sqrt{(\lambda_+-\lambda)(\lambda-\lambda_-)}\cdot\mathbbm{1}_{\{\lambda_-\leq\lambda\leq\lambda_+\}}
\end{equation}
where $\lambda_+ := (1+\gamma^{-1})^2$ and $\lambda_- := (1-\gamma^{-1})^2$. The support of the density, $[\lambda_-, \lambda_+]$, is often called the Mar\v{c}enko-Pastur sea. Regarding the largest eigenvalue $\lambda_{\max}=\lambda_1$ and the smallest eigenvalue $\lambda_{\min} = \lambda_N$ of the sample covariance matrix $S$, German first proved that $\lambda_{\max}\to\lambda_+$ almost surely in \cite{14} and later Silvertein proved that $\lambda_{\min}\to\lambda_-$ almost surely in \cite{15}. That is to say, the largest and the smallest eigenvalues will converge to the corresponding edges of the Mar\v{c}enko-Pastur law. For a second order approximation, Johansson in  \cite{18} proved that $\lambda_{\max}$, properly scaled and centered, will converge weakly to the Tracy-Widom law. Baker, Forrester and Pearce in \cite{20} also proved the similar result for the smallest eigenvalue. More precisely, we have
\begin{eqnarray}
\mathbb{P}\biggl[\Big(\lambda_{\max}-(1+\gamma^{-1})^2\Big)\cdot\frac{\gamma M^{2/3}}{(1+\gamma)^{4/3}}\leq x\biggl] \to F_0(x), \label{l1}\\
\mathbb{P}\biggl[\Big((1-\gamma^{-1})^2-\lambda_{\min}\Big)\cdot\frac{\gamma M^{2/3}}{(\gamma-1)^{4/3}}\leq x\biggl] \to F_0(x). \label{l2}
\end{eqnarray}
Here $F_0(\cdot)$ is the cdf of the Tracy-Widom law, i.e. $F_0(x) := \det(I-\mathcal{A}_x)$ where $\mathcal{A}_x$ is the Airy operator on $L^2(x, \infty)$ with kernel defined later in (\ref{eqn4v2}). Later we will explain why $F_0(\cdot)$  has this zero subscript. We note that most of these results are universal, as is proved in \cite{17}, \cite{23} and \cite{24}, to list a few.

For the spiked population model where $\Sigma\neq I$, the phenomenon becomes much more interesting. Recent research found that the non-null eigenvalues tend to pull the extreme sample eigenvalues out of the Mar\v{c}enko-Pastur sea $[\lambda_-, \lambda_+]$, provided that they are larger or smaller than certain thresholds.  In \cite{3} Baik and Silverstein proved the almost sure limits of the extreme sample eigenvalues pulled out by the spikes. He proved that for fixed $j$, if the largest $j$-th population eigenvalue $\ell_j$ is greater than $1+\gamma^{-1}$, then the largest $j$-th sample eigenvalue $\lambda_j$ satisfies
\begin{equation}\label{conv1}
\lambda_j \to \ell_j + \frac{\ell_j\gamma^{-2}}{\ell_j-1} \quad\text{ almost surely. }
\end{equation}
If $\ell_j$ is less than or equal to the threshold $1+\gamma^{-1}$, then $\lambda_j\to\lambda_+$ which is just the right edge of the Mar\v{c}enko-Pastur sea. The same is true for the smallest eigenvalues. If $\ell_{N-j+1} < 1-\gamma^{-1}$ then (\ref{conv1}) also holds for $j$ replaced by $N-j+1$. Otherwise $\lambda_{N-j+1}\to\lambda_-$ almost surely. Note that this includes the case where some of the $\ell_j$'s are the same. In this case, the corresponding $\lambda_j$'s just converge to the same limit specified in (\ref{conv1}). We call these eigenvalues ``packed''. 

But what is the second order approximation? Baik, Ben Arous and P\'{e}che in \cite{1} observed the phase transition phenomenon of the asymptotic distribution of the largest sample eigenvalue $\lambda_{\max}=\lambda_1$. They proved that if $\ell_1 > 1+\gamma^{-1}$ (i.e., when $\lambda_1$ is pulled out of the sea), then the local fluctuation of $\lambda_1$ will be asymptotically the same as the largest eigenvalue of a $k\times k$ GUE matrix, where $k$ is algebraic multiplicity of $\ell_1$. On the other hand if $\ell_1\leq1+\gamma^{-1}$ then  $\lambda_1$ follows the generalized Tracy-Widom law defined in Definition \ref{def1}. Moreover, in  \cite{4} Bai and Yao obtained the joint local fluctuation of the packed sample eigenvalues which are pulled out --- when suitably centered and scaled, they are asymptotically the same distributed as some unitary Gaussian random matrix. 

We note that similar results can also be obtained for perturbed GUE case, see \cite{2} and \cite{6}. In \cite{4}, \cite{5} and \cite{7} the authors also got the similar results for real Gaussian case or even universal case. 

However, none of these results deal with joint distribution of the largest and the smallest sample eigenvalues in our spiked population model. Intuitively they should be independent. Indeed, the distance between the two extreme eigenvalues is at least $\lambda_+-\lambda_-$, i.e., the width of the Mar\v{c}enko-Pastur sea. For this large distance, the repulsion force between eigenvalues is very weak. As a result, asymptotically they will fluctuate freely, without interacting each other at all.  
For a brief history, in \cite{11} Bianchi, Debbah and Najim first established the independence for the GUE case. Their method is based on bounding the expansion of the Fredholm determinant. 
Later Bornermann in \cite{12} used operator theory to give a much simpler proof. In \cite{19} a different approach relying on PDE can also be found. 

In this paper, the two branches of the results are combined. Here we establish the asymptotic distribution of $(\lambda_1, \lambda_N)$, i.e., the largest and the smallest eigenvalues. More precisely, we have two results. The first one is for the asymptotic distribution of the smallest eigenvalue. Here we also observe the same phase transition phenomenon as in \cite{1}. We proved that, when suitably centered and scaled, the local fluctuation of the smallest eigenvalue will be similar to the largest one discussed above. As a second result, we prove that as $M, N\to\infty$, the largest and the smallest eigenvalues will be asymptotically independent. From this we can easily establish the asymptotic joint distribution of $(\lambda_1, \lambda_N)$. 

\subsection{Main Results}\label{sec1.1}
To state our maim theorems, we need to define a few functions. They  first appeared in \cite{1}. In that paper Baik, Ben Arous and P\'{e}ch\'{e} also discussed some properties of them. 

For any integer $m\geq1$, we define
\begin{eqnarray}
s^{(m)}(u) & = & \frac{1}{2\pi}\int e^{iua + ia^3/3}\frac{1}{(-a)^m}da, \label{eqn1v2}\\
t^{(m)}(v) & = & \frac{1}{2\pi}\int e^{iva + ia^3/3}(-ia)^{m-1}da. \label{eqn2v2}
\end{eqnarray}
Here the contour is from $\infty e^{5\pi/6}$ to $\infty e^{i\pi/6}$ such that the point $a=0$ lies above the contour. As an immediate observation, $s^{(0)} = t^{(1)}$ is just the Airy function, denoted by
\begin{equation}\label{eqn3v2}
\mathrm{Ai}(u) := \frac{1}{2\pi}\int e^{iua+ia^3/3}da
\end{equation}
with the same contour as above. We can also define the Airy kernel
\begin{equation}\label{eqn4v2}
\mathrm{A}(u, v) := \frac{\mathrm{Ai}(u)\mathrm{Ai}'(v)-\mathrm{Ai}'(u)\mathrm{Ai}(v)}{u-v} = \int_0^\infty \mathrm{Ai}(u+y)\mathrm{Ai}(v+y)dy. 
\end{equation}
Furthermore, we define $\mathcal{A}_x$ as the integral operator on $L^2(x, \infty)$ with the kernel $\mathrm{A}(u, v)$. 

%
%
\begin{definition}\label{def1}
Given an integer $k\geq0$, $F_k(x)$ is defined to be 
\begin{equation}\label{eqn5v2}
F_k(x) = \det(1-\mathcal{A}_x)\cdot\det\biggl(\delta_{mn} - \biggl\langle\frac{1}{1-\mathcal{A}_x}s^{(m)}, t^{(n)}\biggl\rangle\biggl)_{m, n = 1}^k
\end{equation}
where $\langle\cdot, \cdot\rangle$ is the inner product in $L^2(x, \infty)$.
\end{definition}
\begin{remark}\label{remark1}
As is shown in \cite{1}, $F_k(\cdot)$ is a distribution function. Moreover, $F_0(\cdot)$ is just the basic Tracy-Widom law. Hence $F_k(\cdot)$ for general $k\geq0$ can be regarded as a generalization of the Tracy-Widom law. 
\end{remark}

%
%
\begin{definition}\label{def2}
For $k\geq1$, we define the distribution function $G_k(x)$ as
\begin{equation}\label{eqn6v2}
G_k(x):= \frac{1}{(2\pi)^{k/2}\prod_{j=1}^kj!}\int_{\infty}^x\ldots\int_{\infty}^x\prod_{i<j}|\xi_i-\xi_j|^2\cdot\prod_{i=1}^ke^{-\xi_i^2/2}d\xi_1\ldots d\xi_k
\end{equation}
to be the distribution function of the largest eigenvalue of a standard $k\times k$ GUE matrix. 
\end{definition}
\begin{remark}
When $k=1$, $G_1(\cdot)$ is just the density of a standard Gaussian. Hence we can regard $G_k(\cdot)$ as a generalization of the Gaussian distribution. 
\end{remark}

Having defined our key distribution functions $F_k(\cdot)$ and $G_k(\cdot)$, we can state our main theorems below. 
%
%
\begin{theorem}\label{thm1}
For the spiked population model defined above, the smallest eigenvalue $\lambda_{\min} = \lambda_N$ of the sample covariance matrix has the following asymptotic distribution, depending on whether the smallest eigenvalue of the true covariance matrix $\Sigma$ is small enough or not.
\begin{enumerate}
\item
If $\ell_N = \ldots = \ell_{N-k+1} = 1-\gamma^{-1}$ for some $k\geq 0$,  and the other $\ell_i$'s lie in a compact subset of $(1-\gamma^{-1}, \infty)$, then as $M\to\infty$ 
\begin{equation}\label{eqn7v2}
\mathbb{P}\biggl\{\frac{\gamma M^{2/3}}{(\gamma-1)^{4/3}}\Big[(1-\gamma^{-1})^2-\lambda_{\min}\Big]\leq x\biggl\}\to F_k(x).
\end{equation}
\item
If $\ell_N = \ldots = \ell_{N-k+1} < 1-\gamma^{-1}$ for some $k\geq 1$,  and the other $\ell_i$'s lie in a compact subset of $(\ell_N, \infty)$, then as $M\to\infty$ 
\begin{equation}\label{eqn8v2}
\mathbb{P}\biggl\{\frac{M^{1/2}}{\sqrt{\ell_N^2-\ell_N^2\gamma^{-2}/(1-\ell_N)^2}}\biggl[\biggl(\ell_N - \frac{\ell_N\gamma^{-2}}{1-\ell_N}\biggl)-\lambda_{\min}\biggl]\leq x\biggl\}\to G_k(x).
\end{equation}
\end{enumerate}
\end{theorem}

Our next theorem states that as $M\to\infty$, the largest and the smallest eigenvalues of the sample covariance matrix are asymptotically independent of each other. 

To state it more rigorously, we assume that $\ell_N = \ldots = \ell_{N-\widehat{k}_1+1}$, $\ell_1 = \ldots = \ell_{\widehat{k}_2}$ for some $\widehat{k}_1, \widehat{k}_2\geq1$, and the rest of the $\ell_i$'s lie in a compact subset of $(\ell_N, \ell_1)$. We define $k_1, k_2$ such that
\begin{equation}\label{k}
k_1 = \left\{
\begin{array}{ll}
0 & \text{ if } \ell_N > (1-\gamma^{-1})^2, \\
\widehat{k}_1 & \text{ if } \ell_N \leq (1-\gamma^{-1})^2.
\end{array}
\right.
\qquad 
k_2 = \left\{
\begin{array}{ll}
0 & \text{ if } \ell_1 < (1+\gamma^{-1})^2, \\
\widehat{k}_2 & \text{ if } \ell_1 \geq (1+\gamma^{-1})^2.
\end{array}
\right.
\end{equation}
For $(\lambda_{\min}, \lambda_{\max}) = (\lambda_N, \lambda_1)$, we define the centered and the scaled version of them. Define
\begin{eqnarray}
\widetilde{\lambda}_{\min} & = & \left\{
\begin{array}{ll}
\frac{\gamma M^{2/3}}{(\gamma-1)^{4/3}}\Big[(1-\gamma^{-1})^2-\lambda_{\min}\Big] & \text{ if } \ell_N \geq (1-\gamma^{-1})^2 \\
\frac{M^{1/2}}{\sqrt{\ell_N^2-\ell_N^2\gamma^{-2}/(1-\ell_N)^2}}\Big[\Big(\ell_N - \frac{\ell_N\gamma^{-2}}{1-\ell_N}\Big)-\lambda_{\min}\Big] & \text{ if } \ell_N < (1-\gamma^{-1})^2
\end{array}
\right.\label{eqn1.10}\\
\widetilde{\lambda}_{\max} & = & \left\{
\begin{array}{ll}
\frac{\gamma M^{2/3}}{(\gamma+1)^{4/3}}\Big[\lambda_{\max}- (1+\gamma^{-1})^2\Big] & \text{ if } \ell_{1} \leq (1+\gamma^{-1})^2 \\
\frac{M^{1/2}}{\sqrt{\ell_1^2+\ell_1^2\gamma^{-2}/(\ell_1-1)^2}}\Big[\lambda_{\max} - \Big(\ell_1 + \frac{\ell_1\gamma^{-2}}{\ell_1-1}\Big)\Big] & \text{ if } \ell_1 > (1+\gamma^{-1})^2
\end{array}
\right.\label{eqn1.11}
\end{eqnarray}
to be the centered and the scaled version of $(\lambda_{\min}, \lambda_{\max})$. Then we have the following theorem regarding the asymptotic distribution of $(\widetilde{\lambda}_{\min}, \widetilde{\lambda}_{\max})$. 

%
%
\begin{theorem}\label{thm2}
Under the assumption of the spiked population model, as $M\to\infty$
\[
\mathbb{P}(\widetilde{\lambda}_{\min}\leq x, \widetilde{\lambda}_{\max}\leq y) \to \left\{
\begin{array}{ll}
F_{k_1}(x)F_{k_2}(y)  & \text{ if } \ell_N\geq (1-\gamma^{-1})^2, \ell_1\leq (1+\gamma^{-1})^2, \\
G_{k_1}(x)F_{k_2}(y)  & \text{ if } \ell_N< (1-\gamma^{-1})^2, \ell_1\leq (1+\gamma^{-1})^2, \\
F_{k_1}(x)G_{k_2}(y)  & \text{ if } \ell_N\geq (1-\gamma^{-1})^2, \ell_1> (1+\gamma^{-1})^2, \\
G_{k_1}(x)G_{k_2}(y)  & \text{ if } \ell_N< (1-\gamma^{-1})^2, \ell_1> (1+\gamma^{-1})^2.
\end{array}
\right.
\]
\end{theorem}
\begin{remark}
Since $F_k(\cdot)$ or $G_k(\cdot)$ is the asymptotic distribution of $\widetilde{\lambda}_{\min}$ and $\widetilde{\lambda}_{\max}$, Theorem \ref{thm2} also tells us that $(\lambda_{\min}, \lambda_{\max})$ are asymptotically independent. 
\end{remark}

Possible applications of our result include data analysis. Suppose we have a large data matrix $X\in\mathbb{C}^{N\times M}$. The question is to determine whether the samples are i.i.d. drawn from a Gaussian distribution of a given spiked covariance matrix $\Sigma$. For example, when $\Sigma = I$ we just want to determine whether our data matrix is pure noise. One possible way is test the hypothesis is to see whether the largest and smallest eigenvalues are close to their asymptotic limit. But here we are testing the null hypothesis using \emph{two} criteria, involving both $\lambda_{\max}$ and $\lambda_{\min}$. Hence we need their joint distribution to carry on the hypothesis testing. Another possible application involves the local fluctuation of condition numbers. They are both discussed in Section \ref{sec2}.  

The rest of the paper is organized as follows. A sketch of the proof will be provided in section \ref{sec1.1}.  In section \ref{sec2} we will discuss some applications of the result. In section \ref{sec3} we will represent the probability $\mathbb{P}(\widetilde{\lambda}_{\min}\leq x, \widetilde{\lambda}_{\max}\leq y)$ in a determinantal form, which will be the basis for future proofs. Section \ref{sec4} and \ref{sec5} will be devoted to prove Theorem \ref{thm1} under the two cases. The proof of Theorem \ref{thm2} can be found in section \ref{sec6}. Finally Section \ref{sec7} serves the conclusion of the whole paper.

\subsection{Sketch of the proof}\label{sec1.3}
First let's focus on the smallest eigenvalue. 
The basic idea is to calculate the large $M$ limit of the  density of $\lambda_{\min}$. Since in our model, the distribution of each sample is Gaussian, we can write down the density explicitly in (\ref{den}). Due to some trick proposed first by Tracy and Widom in \cite{18}, we can further write the density into the form of a Fredholm determinant $\det(I-\mathcal{K}_{11})$, where the kernel of the operator $\mathcal{K}_{11}$ is defined in (\ref{eqn14}). We note that the kernel can be written as two contour integrals. By designing our contours carefully and by the saddle point analysis, we can finally obtain the asymptotic limit of the kernels, or the limit of the operator $\mathcal{K}_{11}$ itself under the trace norm. Since the determinant is a locally Lipschitz function of the operator under the trace norm, we have our desired limit.

For proving the asymptotic independence, our strategy is the same. We write the joint distribution of $(\lambda_{\min}, \lambda_{\max})$ into a determinantal form $\det(I-\mathcal{K})$ in (\ref{eqn4}), where the operator $\mathcal{K} = (\mathcal{K}_{ij})_{i, j = 1}^2$ can be written as a $2\times 2$ block matrix. We prove that under the correct scaling, the diagonal blocks $\mathcal{K}_{11}, \mathcal{K}_{22}$ will have a non-trivial limit, while the off-diagonal blocks will converge to zero, both under the trace norm. Hence $\mathcal{K}$ will converge to a block diagonal matrix. The independence result follows easily from the trivial fact that the determinant of a diagonal matrix equals to the product of the determinants of its diagonal blocks. 

\section{Applications}\label{sec2}
In this section we will discuss some applications of Theorem \ref{thm2}. In subsection \ref{sec2.1} we will talk about the asymptotic distribution of condition numbers of a data matrix. In subsection \ref{sec2.2} some application in hypothesis testing will be discussed.

%
%
\subsection{Distribution of Condition Numbers}\label{sec2.1}
As a simple consequence of Theorem \ref{thm2}, we can get the asymptotic fluctuation of the ratio $\lambda_{\max}/\lambda_{\min}$, which is the square of the condition number of the original data matrix $X$. More precisely, we have the following Corollary. 

%
%
\begin{corollary}\label{cor2.1}
Let $\kappa$ be the condition number of the original data matrix $X$. With the definition of $k_1, k_2$ in (\ref{k}), we have the following four results.
\begin{enumerate}
\item
If $\ell_N\geq (1-\gamma^{-1})^2, \ell_1\leq (1+\gamma^{-1})^2$, then 
\begin{equation}\label{eqn2.1}
M^{2/3}\biggl[\Bigl(\frac{1-\gamma^{-1}}{1+\gamma^{-1}}\Bigl)^2\kappa^2 - 1\biggl] \to\frac{\gamma}{(\gamma-1)^{2/3}}X_{F_{k_1}} +  \frac{\gamma}{(\gamma+1)^{2/3}}Y_{F_{k_2}}.
\end{equation}
\item
If $\ell_N< (1-\gamma^{-1})^2, \ell_1\leq (1+\gamma^{-1})^2$, then 
\begin{equation}\label{eqn2.2}
M^{1/2}\biggl[ \frac{\ell_N-\ell_N\gamma^{-2}/(1-\ell_N)}{(1+\gamma^{-1})^2}\kappa^2-1\biggl] \to
\frac{\sqrt{1-[\gamma(1-\ell_N)]^{-2}}}{1-\gamma^{-2}(1-\ell_N)^{-1}}X_{G_{k_1}}.
\end{equation}
\item
If $\ell_N\geq (1-\gamma^{-1})^2, \ell_1> (1+\gamma^{-1})^2$, then 
\begin{equation}\label{eqn2.3}
M^{1/2}\biggl[\frac{(1-\gamma^{-1})^2}{\ell_1+\ell_1\gamma^{-2}/(\ell_1-1)}\kappa^2 - 1\biggl] \to
\frac{\sqrt{1-[\gamma(\ell_1-1)]^{-2}}}{1+\gamma^{-2}(\ell_1-1)^{-1}}Y_{G_{k_2}}.
\end{equation}
\item
If $\ell_N< (1-\gamma^{-1})^2, \ell_1> (1+\gamma^{-1})^2$, then 
\begin{multline}\label{eqn2.4}
M^{1/2}\biggl[\frac{\ell_N-\ell_N\gamma^{-2}/(1-\ell_N)}{\ell_1+\ell_1\gamma^{-2}/(\ell_1-1)}\kappa^2 - 1\biggl]  \\ \to
\frac{\sqrt{1-[\gamma(1-\ell_N)]^{-2}}}{1-\gamma^{-2}(1-\ell_N)^{-1}}X_{G_{k_1}} + \frac{\sqrt{1-[\gamma(\ell_1-1)]^{-2}}}{1+\gamma^{-2}(\ell_1-1)^{-1}}Y_{G_{k_2}}.
\end{multline}
\end{enumerate}
Here $X, Y$ are two independent random variables with distribution specified in their subscripts. The convergences in (\ref{eqn2.1} -- \ref{eqn2.4}) all mean convergence in distribution. 
\end{corollary}

%
%
\begin{proof}
Here we just prove the corollary for case 1. By $\widetilde{\lambda}_{\min}, \widetilde{\lambda}_{\max}$ in (\ref{eqn1.10}) and (\ref{eqn1.11}) we have
\begin{equation}\label{eqn2.5}
M^{2/3}\biggl[\Bigl(\frac{1-\gamma^{-1}}{1+\gamma^{-1}}\Bigl)^2\frac{\lambda_{\max}}{\lambda_{\min}} - 1\biggl]  =  \frac{(1-\gamma^{-1})}{\lambda_{\min}}\biggl[\frac{\gamma}{(\gamma-1)^{2/3}}\widetilde{\lambda}_{\min} +  \frac{\gamma}{(\gamma+1)^{2/3}}\widetilde{\lambda}_{\max}\biggl].
\end{equation}
In (\ref{eqn2.5}) the first factor will converge  to one in probability. By Theorem \ref{thm2} the random variable in the bracket will convergence to the right hand side of (\ref{eqn2.1}). By Slutsky's lemma our proof is complete. The proofs for the other three cases are the same.
\end{proof}
%
%

%
%
\subsection{Hypothesis Testing}\label{sec2.2}
In this section we discuss one way to test whether the columns of the data matrix $X$ are i.i.d. samples from complex normal distribution of mean zero and covariance matrix $\Sigma$. This is our null hypothesis to be tested. Note that sometimes our goal is just to test whether our data  is just pure noise, where we have $\Sigma = I$. But the method discussed in this subsection can be generalized to any other covariance matrix $\Sigma$ as long as it is of finite rank perturbation of the identity. 

To save space, here we just discuss the case where $\ell_N\geq (1-\gamma^{-1})^2, \ell_1\leq (1+\gamma^{-1})^2$. The method for other cases can also be very similarly obtained. 
By Theorem \ref{thm2}, we know that $\lambda_{\max}, \lambda_{\min}$ will converge to $(1+\gamma^{-1})^2$ and $(1-\gamma^{-1})^2$, respectively. Intuitively, if the observed two extreme eigenvalues of the sample covariance matrix deviates from the two limits by a large amount, then probably we should reject the null hypothesis. However, here we are testing the null hypothesis using two criteria. We claim that $\lambda_{\min}$ should be close to $(1-\gamma^{-1})^2$ \emph{and} $\lambda_{\max}$ should be close to $(1+\gamma^{-1})^2$. Thus we need to calculate the joint distribution of the two. Hence our Theorem \ref{thm2} plays an important role here --- it claims that they are asymptotically independent. 

Here's our method to test the null hypothesis, step by step.
\begin{enumerate}
\item
Form the sample covariance matrix $S = XX^*/M$ from the data matrix $X$. Then calculate the observed maximum and minimum eigenvalue of it, denoted by $(\lambda_{\max}, \lambda_{\min})$.
\item
Use (\ref{eqn1.10}) and (\ref{eqn1.11}) to calculate the centered and scaled version of $(\lambda_{\max}, \lambda_{\min})$, denoted by $(\widetilde{\lambda}_{\max}, \widetilde{\lambda}_{\min})$.
\item
Calculate our statistic $T := (1-F_{k_1}(\widetilde{\lambda}_{\min}))(1-F_{k_2}(\widetilde{\lambda}_{\max}))$. If $T\leq \alpha$ where $\alpha$ is a pre-determined confidence level, then we reject the hypothesis. Otherwise we accept it. 
\end{enumerate}

\section{Determinantal Form of the Distribution}\label{sec3}
In this section we will derive the determinantal form of the distribution of the extreme eigenvalues. This formula will be the basis for future saddle point analysis from Section \ref{sec4} to \ref{sec6}. From now on, we will always denote $C, c$ as some positive constants independent of $M, N$. Their value may vary from line to line, but they are always constants. 

For notational convenience, we define $\pi_i = \ell_i^{-1}$ for $i = 1, \ldots, N$. Then if $\pi_i$'s are distinct, the joint probability density of $(\lambda_1, \ldots, \lambda_N)$ can be written as the following form
\begin{equation}\label{den}
p(\lambda_1, \ldots, \lambda_N) = \frac{1}{C}\det\Bigl(\exp(-M\pi_j\lambda_k)\Bigl)_{j, k = 1}^N\cdot\frac{V(\lambda)}{V(\pi)}\prod_{j=1}^N\lambda_j^{M-N}.
\end{equation}
for some constant $C>0$. Here $V(x) := \prod_{i<j}(x_i-x_j)$ is the Vandermonde determinant. If some of the $\pi_j$'s coincide, we interpret (\ref{den}) using the L'H\'{o}pital's rule. 

For any two real numbers $0<\xi_1<\xi_2<\infty$, let $J_1, J_2$ be two disjoint intervals on $\mathbb{R}$ such that  
\[
J_1 = (0, \xi_1), \qquad J_2 = (\xi_2, \infty).
\]
Moreover, we define $J:= J_1\sqcup J_2$ and we denote $E(J, 0)$ as the probability that there are no eigenvalues in the set $J$. That is, 
\[
E(J, 0) = \mathbb{P}(\xi_1\leq\lambda_{\min}\leq\lambda_{\max}\leq\xi_2).
\]
Our first proposition gives a determinantal representation of $E(J, 0)$. This determinant form will turn out to be an invaluable formula for future analysis.

%
%
\begin{proposition}\label{prop1}
We have 
\begin{equation}\label{eqn4}
E(J, 0) = \det\biggl(I - 
\left(
\begin{array}{cc} 
K_{11} & K_{12} \\
K_{21} & K_{22}\\ 
\end{array}
\right)
\biggl)
\end{equation}
where for $\alpha, \beta \in\{1, 2\}$ $K_{\beta, \alpha}$ is the integral operator on $L^2(J_\alpha)\to L^2(J_\beta)$ with the kernel
\begin{equation}\label{eqn5}
K_{\beta, \alpha}(\eta, \zeta)  =  
-\frac{M}{(2\pi)^2}\int_{\Gamma_\beta} dz\int_{\Sigma_\alpha} dw e^{-\eta M(z-q_\beta)+\zeta M(w-q_\alpha)}\frac{1}{w-z}\prod_{k=1}^N\frac{w-\pi_k}{z-\pi_k}\cdot\biggl(\frac{z}{w}\biggl)^M. 
\end{equation}
Here $q_1, q_2$ are two pre-fixed constants such that 
\begin{equation}\label{con}
0 < q_2 < \min\{\pi_j\}_{j=1}^N \leq \max\{\pi_j\}_{i=1}^N < q_1. 
\end{equation}
Moreover $\Gamma_{\beta} (\beta = 1, 2)$ is the contour on $\mathbb{C}$ enclosing $\{\pi_i\}_{i=1}^N$ and lying in the strip
\[
\Gamma_{\beta} \subseteq \{z\in\mathbb{C}: q_2< \Re(z)<q_1\}. 
\]
$\Sigma_1 = \{z = A+it: t\in\mathbb{R}\}$ is a vertical line from bottom to top for $A > q_1$. $\Sigma_2$ is counter clockwise circle enclosing the origin and lies in $\Sigma_2\subseteq\{z: \Re(z)<q_1\}$. These contours are sketched in Figure \ref{fig1}. 
\end{proposition}
\begin{figure}[!h]
    \centering
    \includegraphics[scale=.6]{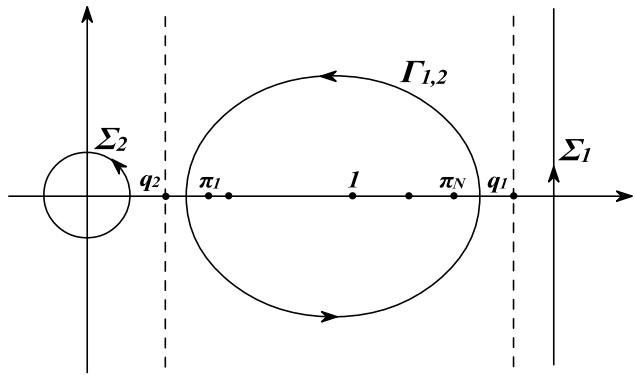}
    \caption{Graph for the contours $\Sigma_1, \Sigma_2$ and $\Gamma$}
    \label{fig1}
\end{figure}
\begin{remark}
Here $q_1, q_2$ can be any  constants as long as (\ref{con}) is satisfied. They will not affect the determinant in (\ref{eqn4}). 
\end{remark}

%
%
\begin{proof}
The proof is almost the same to the one in \cite{1}. Thus we just point out our major differences. The rest will be quite sketchy.
Since we have an explicit expression for the probability density (\ref{den}), we can express $E(J, 0)$ into the following (for details, see (66) of \cite{1}). 
\begin{eqnarray*}
E(J, 0) & = & \frac{1}{C}\int_0^\infty\ldots\int_0^\infty\det_{j, k = 1}^N(\lambda_k^{j-1})\det_{j, k = 1}^N(e^{-M\pi_j\lambda_k})\prod_{k=1}^N(1-\chi_J(\lambda_k))\lambda_k^{M-N}d\lambda_k \\
& = & \frac{1}{C}\det\biggl(\int_0^\infty(1-\chi_J(\lambda))\lambda^{j-1+M-N}e^{-M\pi_k\lambda}d\lambda\biggl)_{j, k = 1}^N \\
& = & \frac{1}{C}\det\biggl(A_{jk} - \int_J\lambda^{j-1+M-N}e^{-M\pi_k\lambda}d\lambda\biggl)_{j, k = 1}^N.
\end{eqnarray*}
Here $\chi_J(\cdot)$ is the indicator function of the set $J$ and $A_{jk}$ is defined to be
\[
A_{jk} := \int_0^\infty\lambda^{j-1+M-N}e^{-M\pi_k\lambda}d\lambda
\]
If we define $A = (A_{jk})_{j, k=1}^N$, then we can show that 
\[
\det A = \prod_{j=1}^N\frac{\Gamma(j+M-N)}{(M\pi_j)^{M-N}}\cdot\prod_{1\leq j<k\leq N}\biggl(\frac{1}{M\pi_j}-\frac{1}{M\pi_k}\biggl).
\]
Now we define $\phi_j^{(i)}(\lambda) = \lambda^{j-1+M-N}e^{-Mq_i\lambda}$ and $\Phi_k^{(i)}(\lambda) = e^{-M(\pi_k-q_i)\lambda}$ for $i = 1, 2$. Here $q_1, q_2$ are two pre-fixed constant such that condition (\ref{con}) is satisfied.
Then we observe $\phi^{(i)}_j(\lambda), \Phi^{(i)}_j(\lambda)\in L^2(J)$. 
We also define the operators $F: L^2(J)\to\mathbb{R}^N, G: \mathbb{R}^N\to L^2(J)$ by 
\begin{eqnarray*}
F:  f(\lambda) & \mapsto & \biggl(\int_{J_1}\phi^{(1)}_j(\lambda)f(\lambda)d\lambda+\int_{J_2}\phi^{(2)}_j(\lambda)f(\lambda)d\lambda\biggl)_{j=1}^N \\
G: \alpha\in\mathbb{R}^N & \mapsto & \sum_{j=1}^N\alpha_j\Phi^{(1)}_j(\lambda)\chi_{J_1}(\lambda)+\sum_{j=1}^N\alpha_j\Phi^{(2)}_j(\lambda)\chi_{J_2}(\lambda).
\end{eqnarray*}
Then we can re-write $E(J, 0)$ as 
\begin{eqnarray}
E(J, 0) & = & \frac{1}{C}\det(A - FG) = \frac{\det A}{C}\det(I-A^{-1}FG) \nonumber\\
& = & C'\det(I - GA^{-1}F). \label{eqn1}
\end{eqnarray}
Please note that $GA^{-1}F$ is an operator on $L^2(J)\to L^2(J)$ and $C'$ is a constant independent of $J$. We note that $GA^{-1}F = P_JGA^{-1}FP_J$ can be also regarded as an operator on $L^2(0, \infty)$. Recall that $J = (0, \xi_1)\sqcup(\xi_2, \infty)$. With $\xi_1 \to 0$ and $\xi_2\to\infty$, $GA^{-1}F$ will converge to zero. Hence take $J\to\emptyset$ in (\ref{eqn1}) gives $C' = 1$. 

Now we note that $L^2(J) = L^2(J_1\sqcup J_2)$ is isomorphic to $L^2(J_1)\oplus L^2(J_2)$ under the trivial isomorphism 
\[
f\in L^2(J_1\sqcup J_2) \mapsto (f|_{J_1}, f|_{J_2}) \in L^2(J_1)\oplus L^2(J_2).
\]
Thus, under the isomorphism, (\ref{eqn1}) can be further written as 
\begin{equation}\label{eqn2}
E(J, 0) = \det\biggl(I - 
\left(
\begin{array}{cc} 
P_{J_1}GA^{-1}FP_{J_1} & P_{J_1}GA^{-1}FP_{J_2} \\
P_{J_2}GA^{-1}FP_{J_1} & P_{J_2}GA^{-1}FP_{J_2}\\ 
\end{array}
\right)
\biggl)
\end{equation}
Here $P_{J_\beta}GA^{-1}FP_{J_\alpha}$ is an operator on $L^2(J_\alpha)\to L^2(J_\beta)$.  Following the rest of the proof of Proposition 2.1 in \cite{1}, we can obtain that, for $\alpha, \beta\in\{1, 2\}$,
\begin{multline}
(P_{J_\beta}GA^{-1}FP_{J_\alpha})(\eta, \zeta)  =  \\
-\frac{M}{(2\pi)^2}\int_{\Gamma_\beta} dz\int_{\Sigma_\alpha} dw e^{-\eta M(z-q_\beta)+\zeta M(w-q_\alpha)}\frac{1}{w-z}\prod_{\ell=1}^N\frac{w-\pi_k}{z-\pi_\ell}\cdot\biggl(\frac{z}{w}\biggl)^M. \label{m1}
\end{multline}
Here $\Gamma_\beta$ and $\Sigma_\alpha$ are the contours defined in Proposition \ref{prop1}. The only difference is that in our case, instead of using equation (76) in \cite{1}, we used 
\[
\lambda^{j-1+M-N} = \frac{\Gamma(j+M-N)}{2\pi i}\int_{\Sigma_\beta}e^{\lambda Mw}\frac{M}{(Mw)^{j+M-N}}dw
\]
for both $\beta = 1, 2$, as long as $M-N = (\gamma-1)N$ is large enough. 
\end{proof}
%
%

If we consider the special case of Proposition \ref{prop1} when $J_1 = (0, \xi)$ and $J_2 = \emptyset$, then we arrive at the following Corollary.
%
%
\begin{corollary}\label{corollary1}
We have 
\begin{equation}\label{eqn6}
\mathbb{P}(\lambda_{\min}\geq\xi) = \det(I - 
K_{11})
\end{equation}
where $K_{11}$ is the integral operator on $L^2(0, \xi)\to L^2(0, \xi)$ with the kernel defined in (\ref{eqn5}).
\end{corollary}

As the first goal of this paper, we will focus on the distribution of the smallest eigenvalue $\lambda_{\min}$. Due to Corollary \ref{corollary1}, we can in turn analyze the asymptotic behavior of the operator $K_{11}$. 

First we note that the kernel function $K_{11}(\eta, \zeta)$ can be simplified. Since the contour $\Gamma$ is always on the left of $\Sigma_1$, then $\forall z\in\Gamma_1, w\in\Sigma_1$ we have $\Re(z-w)<0$. Hence we have 
\begin{equation}\label{eqn7}
\frac{1}{w-z} = M\int_0^\infty\exp\biggl\{-yM(w-q_1-(z-q_1))\biggl\}dy.
\end{equation}
Substitute (\ref{eqn7}) into (\ref{eqn5}), we obtain 
\begin{equation}\label{eqn8}
K_{11}(\eta, \zeta) = -\int_0^\infty H_1(y-\eta)J_1(y-\zeta)dy
\end{equation}
where
\begin{eqnarray}
H_1(\eta) &=& \frac{M}{2\pi}\int_{\Gamma_1} \exp\biggl\{\eta M(z-q_1)\biggl\}\cdot z^M\prod_{\ell=1}^N\frac{1}{z-\pi_\ell}dz, \label{eqn9}\\
J_1(\zeta) & = & \frac{M}{2\pi}\int_{\Sigma_1} \exp\biggl\{-\zeta M(w-q_1)\biggl\}\cdot \omega^{-M}\prod_{\ell=1}^N(\omega - \pi_\ell)d\omega\label{eqn10}. 
\end{eqnarray}

Now we use the change of variables. Define
\begin{equation}\label{eqn11}
\xi = \mu_1 - \frac{\nu_1}{M^{\alpha}}x, \qquad \eta = \mu_1 - \frac{\nu_1}{M^{\alpha}}(x+u), \qquad \zeta = \mu_1 - \frac{\nu_1}{M^{\alpha}}(x+v)
\end{equation}
where $\mu_1, \nu_1$ and $\alpha$ are constants that will be defined later. Under such change of variables 
\[
\mathbb{P}(\lambda_{\min}\geq\xi) = \det(I-K_{11}|_{(0, \xi)}) = \det(I-K_{11}|_{(0, \mu_1-\nu_1x/M^\alpha)}) = \det(I-\mathcal{K}_{11}|_{(0, \mu_1M^{\alpha}/\nu_1-x)})
\]
where
\begin{eqnarray}
\mathcal{K}_{11}(u, v) & = & \frac{\nu_1}{M^\alpha}K_{11}\biggl(\mu_1-\frac{\nu_1}{M^{\alpha}}(x+u), \mu_1-\frac{\nu_1}{M^{\alpha}}(x+v)\biggl) \nonumber\\
& = & -\int_0^\infty\mathcal{H}_1(y+x+u)\mathcal{J}_1(y+x+v)dy \label{eqn14}.
\end{eqnarray}
Here 
\begin{eqnarray}
\mathcal{H}_1(u)& =& \frac{\nu M^{1-\alpha}}{2\pi}\int_\Gamma e^{-\mu_1M(z-q)+\nu M^{1-\alpha}u(z-q_1)}\cdot z^M\prod_{\ell=1}^N\frac{1}{z-\pi_\ell}dz,\label{eqn12} \\
\mathcal{J}_1(v) &=& \frac{\nu M^{1-\alpha}}{2\pi}\int_{\Sigma_1}e^{\mu_1M(\omega-q)-\nu M^{1-\alpha}v(\omega-q_1)}\cdot \omega^{-M}\prod_{\ell=1}^N(\omega-\pi_\ell)d\omega.\label{eqn13}
\end{eqnarray}
Under the assumption that only $r_1 + r_2$ of these $\pi_k$'s are not one, we can further write $\mathcal{H}_1(u), \mathcal{J}_1(v)$ as 
\begin{eqnarray}
\mathcal{H}_1(u)& =& \frac{\nu M^{1-\alpha}}{2\pi}\int_\Gamma e^{\nu M^{1-\alpha}u(z-q_1)+Mf_1(z)}\cdot \prod_{\ell\leq r_2 \text{ or } \ell\geq N-r_1+1}\frac{z-1}{z-\pi_\ell}dz,\label{eqn15} \\
\mathcal{J}_1(v) &=& \frac{\nu M^{1-\alpha}}{2\pi}\int_{\Sigma_1}e^{-\nu M^{1-\alpha}v(\omega-q_1)-Mf_1(\omega)}\cdot \prod_{\ell\leq r_2 \text{ or } \ell\geq N-r_1+1}\frac{\omega-\pi_\ell}{\omega-1}d\omega.\label{eqn16}
\end{eqnarray}
Here
\begin{equation}\label{eqn19}
f_1(z) := -\mu_1(z-q_1)+\log z - \log(z-1)/\gamma^2
\end{equation}
where the $\log$ function is defined in the principal branch. 
According to the discussion in \cite{1}, we just need to find suitable limiting functions $\mathcal{H}_{1, \infty}(u), \mathcal{J}_{1, \infty}(v)\in L^2(0, \infty)$ and a constant $Z_M$ such that, for any $x$ in a compact set,
\begin{eqnarray}
\int_0^\infty\int_0^\infty |Z_M\mathcal{H}_1(x+u+y)-\mathcal{H}_{1, \infty}(x+u+y)|^2dudy\to 0, \label{eqn17}\\
\int_0^\infty\int_0^\infty |Z_M^{-1}\mathcal{J}_1(x+u+y)-\mathcal{J}_{1, \infty}(x+u+y)|^2dudy\to 0. \label{eqn18}
\end{eqnarray}
In the following two sections we will use the saddle point analysis on $f_1(z)$ to realize our promise. 

\section{Proof of Part 1 of Theorem \ref{thm1}}\label{sec4}
In this section we will prove Theorem \ref{thm1} for the case $\ell_N\geq 1-\gamma^{-1}$. Since in this section and the next section we will focus on only $\mathcal{K}_{11}$ defined above, thus we temporarily drop all the subscript one, for notational simplicity. For example, instead of writing $\mu_1, \nu_1, \Gamma_1, f_1(z), \mathcal{H}_1(u), \mathcal{H}_{1, \infty}(u)$, we will write $\mu, \nu, \Gamma, f(z), \mathcal{H}(u), \mathcal{H}_\infty(u)$. Note that these are only valid in this and the next section.

We assume that for some $0\leq k\leq r_1$ (recall $\pi_i^{-1} = \ell_i$), 
\[
\pi_{N}^{-1} = \pi_{N-1}^{-1} = \ldots = \pi_{N-k+1}^{-1} = 1-\gamma^{-1} 
\]
and the rest of the $\pi^{-1}_i$'s are in a compact subset of $(1-\gamma^{-1}, \infty)$. The rest of the job is to provide a saddle point analysis of $\mathcal{H}(u)$ and $\mathcal{J}(v)$. 

In this case, we can define our constants to be
\begin{equation}\label{eqn20}
\alpha = 2/3, \quad\mu = \biggl(1-\frac{1}{\gamma}\biggl)^2, \quad \nu = \frac{(\gamma-1)^{4/3}}{\gamma}, \quad p = \frac{\gamma}{\gamma-1}, \quad q = p + \frac{\epsilon}{\nu M^{1/3}}. 
\end{equation}
Here $\epsilon>0$ is a pre-fixed smaller number and  $p$ will be later shown to be the saddle point of $f(z)$. We also define 
\begin{equation}\label{eqn21} 
g(z) = \frac{1}{(z-1)^{r_1+r_2-k}}\prod_{\ell\leq r_2\text{ or } N-r_1+1\leq\ell\leq N-k}(z-\pi_\ell). 
\end{equation}
Then our $\mathcal{H}(u), \mathcal{J}(v)$ can be written as 
\begin{eqnarray}
\mathcal{H}(u)& =& \frac{\nu M^{1-\alpha}}{2\pi}\int_\Gamma e^{\nu M^{1-\alpha}u(z-q)}\cdot e^{Mf(z)}\cdot \frac{1}{g(z)(z-p)^k}dz,\label{eqn21.1} \\
\mathcal{J}(v) &=& \frac{\nu M^{1-\alpha}}{2\pi}\int_{\Sigma}e^{-\nu M^{1-\alpha}v(\omega-q)}\cdot e^{-Mf(\omega)}\cdot g(\omega)(\omega-p)^kd\omega.\label{eqn22}
\end{eqnarray}
Our constant $Z_M$ in this case is defined to be 
\begin{equation}\label{eqn23}
Z_M = \frac{g(p)}{(\nu M^{1/3})^ke^{Mf(p)}}.
\end{equation}
Finally, as we promised in (\ref{eqn15}) and (\ref{eqn16}), our $\mathcal{H}_\infty(u)$ and $\mathcal{J}_\infty(y)$ are defined as 
\begin{eqnarray}
\mathcal{H}_\infty(u)& = & \frac{e^{-\epsilon u}}{2\pi}\int_{\Gamma_{\infty}} e^{ua-a^3/3}\frac{1}{a^k}da, \label{eqn29} \\
\mathcal{J}_\infty(v) & = & \frac{e^{\epsilon u}}{2\pi}\int_{\Sigma_{\infty}} e^{-vb+b^3/3}b^kdb.\label{eqn30}
\end{eqnarray}
Here $\Gamma_{\infty}$ is the contour from $\infty e^{-2\pi/3}$ to $\infty e^{2\pi/3}$ such that $a=0$ lies on the left hand side of the contour. $\Sigma_{\infty}$ is the contour from $\infty e^{-\pi/3}$ to $\infty e^{\pi/3}$. 
The main target of this section is to prove the following Proposition.
%
%
\begin{proposition}\label{prop2}
Assume that $u$ and $v$ are bounded below. That is, assume that there exists some $U, V\in\mathbb{R}$ such that $u>U, v>V$. Then 
\begin{enumerate}
\item
There exists some constant $C, c>0, M_0>0$ such that uniformly for $u>U$, 
\begin{equation}\label{eqn31}
|Z_M\mathcal{H}(u)-\mathcal{H}_\infty(u)|\leq Ce^{-cu}M^{-1/3},  \qquad \forall M\geq M_0.
\end{equation}
\item
There exists some constant $C, c>0, M_0>0$ such that uniformly for $v>V$,
\begin{equation}\label{eqn32}
|Z_M^{-1}\mathcal{J}(v)-\mathcal{J}_\infty(v)|\leq Ce^{-cv}M^{-1/3}, \qquad \forall M\geq M_0.
\end{equation}
\end{enumerate}
\end{proposition}

In order to prove Proposition \ref{prop2}, we perform the saddle point analysis on $f(z)$ defined in (\ref{eqn19}) by $f(z) := -\mu(z-q) +\log z-\gamma^{-1}\log(z-1)$. Through some very simple calculation, we find that 
\begin{equation}\label{eqn24}
f'(p) = f''(p) = 0, \quad  f'''(p) = -2\nu^3.
\end{equation}
Hence $p$ is the saddle point of $f(z)$. Since $f'''(p)$ is negative, the steepest descent direction for $e^{Mf(z)}$ appearing in $\mathcal{H}(u)$ is with angle $\pm 2\pi/3$ to the real axis. The steepest descent direction for $e^{-Mf(\omega)}$ in $\mathcal{J}(v)$ is with angle $\pm \pi/3$ to the real axis. Let's analyze $\mathcal{H}(u)$ first.  

%
%
\subsection{Analysis of $\mathcal{H}(u)$}
To carry out an saddle point analysis, we need to define our contour $\Gamma$ for $\mathcal{H}(u)$ first. In the following we only define $\Gamma$ in the upper half of the complex plane. The part in the lower half is just of a reflection with respect to the real axis. Let $\Gamma:=\bigcup_{i=0}^5\Gamma_i\cup\overline{\bigcup_{i=0}^5\Gamma_i}$ where 
\begin{eqnarray*}
\Gamma_0 & = & \biggl\{z = p + \frac{\epsilon}{2\nu M^{1/3}}e^{i\theta} : 0\leq\theta\leq\frac{2}{3}\pi\biggl\}, \\
\Gamma_1 & = & \biggl\{z = p + e^{2\pi i/3}t: \frac{\epsilon}{2\nu M^{1/3}}\leq t\leq (\sqrt{3}-1)(p-1)\biggl\}
\end{eqnarray*}
where $\epsilon$ is the small constant appeared in the definition for $q$ in (\ref{eqn20}). Define $\Gamma_2$ to be the arc centered at the point $z=1$ and connects the endpoint of $\Gamma_1$. We define 
\begin{equation}\label{eqn25}
p^* = \frac{\gamma}{1+\gamma}.
\end{equation}

\begin{figure}[!h]
    \centering
    \includegraphics[scale=.7]{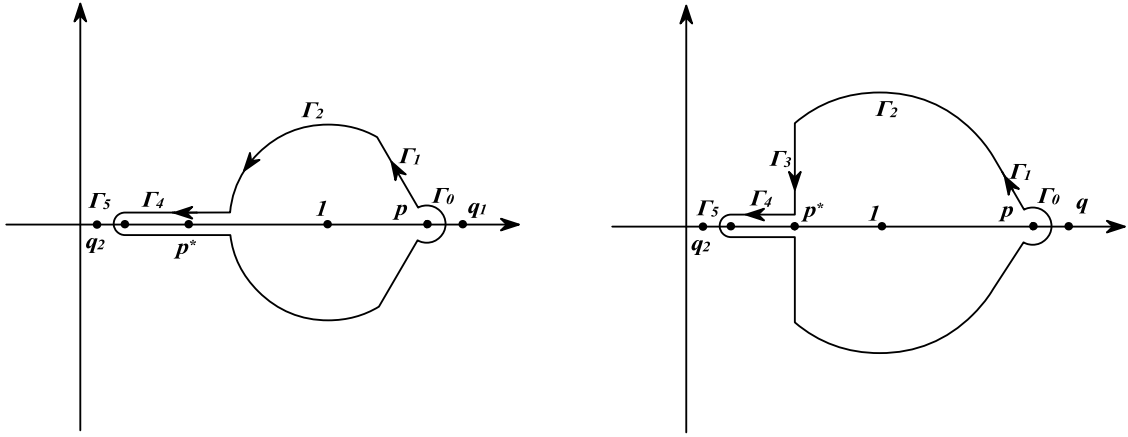}
    \caption{The Contours $\Gamma$. Case 1: Left. Case 2: Right.}
    \label{fig2}
\end{figure}
For the definition of the rest of the contours, we need to split it into two cases.
\begin{itemize}
\item
Case 1. If $\Gamma_2$ does not intersect the vertical line $\{z = p^* + it: t\in\mathbb{R}\}$, then $\Gamma_4$ is defined to be a horizontal line just above the real axis, i.e.
\begin{equation}\label{eqn26}
\Gamma_4 \subseteq \{z = t + i\kappa: t\in\mathbb{R}^+\}
\end{equation}
for some small fixed constant $\kappa > 0$. We also require that $\Gamma_4$ should connect the endpoint of $\Gamma_2$ and the point $\min\{\pi_1, p^*\} + i\kappa$. $\Gamma_5$ is a quarter circle centering at $\min\{\pi_1, p^*\} + i\kappa$, with radius $\kappa$, just connecting $\Gamma_4$. For a graph of the contour, please refer to the left part of Figure \ref{fig2}.
\item
Case 2. If $\Gamma_2$ intersects the vertical line $\{z = p^* + it: t\in\mathbb{R}\}$, then we just define $\Gamma_3$ to be the part of that line, connecting the endpoint of $\Gamma_2$ and heading downward. We add $\Gamma_3$ into our contour. The definition for $\Gamma_4$ and $\Gamma_5$ are exactly the same as that in Case 1. For a graph for this case, please refer to the right part of Figure \ref{fig2}. 
\end{itemize}

We choose the contour $\Gamma = \bigcup_{i=0}^5\Gamma_i\cup\overline{\bigcup_{i=1}^5\Gamma_i}$ in such a specific way to establish the following lemma. 
%
%
\begin{lemma}\label{lemma1}
In both cases, $\Re(f(z))$ is monotonously decreasing when $z$ travels along the path $\bigcup_{i=1}^3\Gamma_i$, provided that we choose $\kappa$ small enough. For $\Gamma_4, \Gamma_5$ we have 
\[
\sup_{z\in\Gamma_4\sqcup\Gamma_5}\Re(f(z)) < \Re(f(z^*))
\]
where $z^*$ is the intersection of $\Gamma_1$ and $\Gamma_2$.  
\end{lemma}

%
%
\begin{proof}
We divide the proof into several parts. 
\begin{itemize}
\item
In $\Gamma_1, z = p + te^{2\pi i/3}$ for $(2\nu M^{1/3})^{-1}\epsilon\leq t\leq (\sqrt{3}-1)(p-1)$. Then 
\[
\frac{d}{dt}\Re(f(z)) = \frac{(\gamma-1)^4t^2[t^2+(2p-1)t-2p(p-1)]}{2\gamma^2[t^2+(1-p)t+(1-p)^2][t^2-pt+p^{2}]}.
\]
Everything is positive except the quadratic term $t^2+(2p-1)t-2p(p-1)$. When $t$ is taken to be its largest possible value $(\sqrt{3}-1)(p-1)$, the quadratic term achieves it maximum, which is $-\sqrt{3}(\sqrt{3}-1)(p-1)<0$. Hence we get $d\Re(f(z))/dt <0.$
\item
In $\Gamma_2$ we have $z = 1+Re^{-i\theta}$ for $\theta$ taking value in a subset of $(0, \pi)$. The $R$ is some constant. We have
\begin{eqnarray}
\frac{d}{d\theta}\Re(f(z))& =& \frac{R\sin\theta}{1+R^2+2R\cos\theta}\biggl[\mu(1+R^2+2R\cos\theta)-1\biggl] \nonumber\\
& \leq & \frac{R\sin\theta}{1+R^2+2R\cos\theta}\biggl[\mu(R+1)^2-1\biggl].\label{eqn27}
\end{eqnarray}
But we can calculate $R = (p-1)\cdot(3-\sqrt{3})/\sqrt{2} < p-1 = 1/\mu^2-1$. Hence $d\Re(f(z))/d\theta<0$.
\item
In $\Gamma_3$, we note that $z\in\Gamma_3$ can be represented as $z = \gamma/(1+\gamma) + it$ for some real parameter $t\in\mathbb{R}^+$. We then have 
\[
\frac{d}{dt}\Re(f(z)) = -\frac{(\gamma+1)^4(\gamma^2-1)t^2}{\gamma^2\Big[\gamma^2+(\gamma+1)^2t^2\Big]\Big[1+(\gamma+1)^2t^2\Big]}>0.
\]
Since $t$ is decreasing as $z$ travels along $\Gamma_3$, we conclude that $\Re(f(z))$ is decreasing. 
\item
For $z\in\Gamma_4\sqcup\Gamma_5$, $z$ can be represented as $z =t+\kappa i$ for some small constant $\kappa>0$ and real parameter $t$. 
We note that for $\kappa = 0,$
\[
\frac{d}{dt}\Re(f(z))  =   \frac{\mu(t-p^*)^2}{\gamma^2t(1-t)} >0, \qquad\forall t\in(0, 1). 
\]
Moreover  $\Re(f(z))$ is continuous for $\kappa$ around zero, uniformly for all  $t$ if $t$ lies in a compact subinterval of $(0, 1)$, which is indeed the case if $z\in\Gamma_4\sqcup\Gamma_5$. Hence we can choose $\kappa$ small enough so that 
\[
\sup_{z\in\Gamma_4\sqcup\Gamma_5}\Re(f(z)) < \Re(f(z^*)).
\]
\end{itemize}. 
\end{proof}
%
%

Now we are ready to find an asymptotic expression for the integral 
\begin{equation}\label{eqn28}
Z_M\mathcal{H}(u) = \frac{\nu M^{1/3}}{2\pi}\int_\Gamma e^{\nu M^{1/3}u(z-q)}\cdot e^{M[f(z)-f(p)]}\cdot\frac{g(p)}{g(z)}\cdot\frac{1}{[\nu M^{1/3}(z-p)]^k}dz
\end{equation}
As in \cite{1}, we first fix some $\delta>0$ small, then we decompose $\Gamma = \Gamma' \sqcup\Gamma''$ where 
\begin{equation}\label{eqn33}
\Gamma' := \{z\in\Gamma: |z-p|<\delta\}, \qquad \Gamma'' := \Gamma \backslash \Gamma'.
\end{equation}
We can also define $\Gamma_{\infty} = \Gamma'_{\infty}\sqcup\Gamma''_{\infty}$ where $\Gamma'_{\infty} := \nu M^{1/3}(\Gamma'-p)$ is the image of $\Gamma'$ under the map, and $\Gamma_{\infty}'' := \Gamma_\infty\backslash\Gamma'_{\infty}$. We can also set
\begin{equation}\label{eqn34}
\mathcal{H}(u) = \mathcal{H}'(u) + \mathcal{H}''(u), \qquad \mathcal{H}_{\infty}(u) = \mathcal{H}_{\infty}'(u) + \mathcal{H}_{\infty}''(u).  
\end{equation}
Here $\mathcal{H}'(u)$ is the part of the integral (\ref{eqn28}) on $\Gamma'\cup\overline{\Gamma'}$ and the definitions for $\mathcal{H}''(u), \mathcal{H}'_{\infty}(u)$ and $\mathcal{H}''_{\infty}(u)$ in (\ref{eqn34}) are similar. 

Our strategy to prove (\ref{eqn31}) in Proposition \ref{prop1} is as follows. 
\begin{itemize}
\item
By saddle point analysis, the integral of $z$ in $\mathcal{H}(u), \mathcal{H}_\infty(u)$ should be concentrated on the point $z = p$. Hence $\mathcal{H}''(u)$ and $\mathcal{H}''_\infty(u)$ should be negligible, as is proved in Lemma \ref{lemma2}.
\item
For the integral of $z$ near $z=p$, we prove in Lemma \ref{lemma3} that the difference $|\mathcal{H}'(u)-\mathcal{H}_\infty'(u)|$ is small.
\end{itemize}

Now we discuss how we should choose $\delta$. First, by some very basic calculus we obtain that, if $\delta < (p-1)/2$, then for $|z-p|<\delta$ we have
\begin{equation}\label{36}
\frac{1}{4!}|f^{(4)}(z)| \leq \frac{1}{4} + \frac{4}{\gamma^2(p-1)^4} := C_0. 
\end{equation}
We then choose $\delta >0$ such that 
\begin{equation}\label{eqn37}
\delta < \min\biggl\{\frac{\nu^3}{6C_0}, \frac{p-1}{2}\biggl\}. 
\end{equation}
For such choice of $\delta$, we  have, for $|z-p|<\delta$
\begin{eqnarray}
\biggl|f(z)-f(p)-\frac{f'''(p)}{3!}(z-p)^3\biggl| &\leq & \max_{|s-p|\leq\delta}\frac{|f^{(4)}(s)|}{4!}|z-p|^4  \nonumber \\
& \leq & C_0|z-p|^4 \leq \frac{\nu^3}{6}|z-p|^3. \label{eqn38}
\end{eqnarray}
Therefore, for $z\in\Gamma_1\cap\Gamma'$, $z = p + te^{2i\pi/3}$ for $0<t<\delta$. Hence (recall $f'''(p) = -2\nu^3$ in (\ref{eqn24}))
\begin{equation}\label{eqn39}
\Re({f(p+te^{2i\pi/3}}) - f(p))\leq -\frac{\nu^3}{6}t^3.
\end{equation}
Also, because of Lemma $\ref{lemma1}$, we know that $\Re(f(z)-f(p))$ for $z\in\Gamma''$ is even smaller. Hence we have 
\begin{equation}\label{eqn41}
\Re(f(z) - f(p))\leq -\nu^3\delta^3/6, \qquad \forall z\in \Gamma''.
\end{equation}
because of this, intuitively $\mathcal{H}''(u)$ and $\mathcal{H}''_{\infty}$ defined in (\ref{eqn34}) are negligible, as the following lemma states. 
%
%
\begin{lemma}\label{lemma2}
If $u$ is bounded below, i.e., $u>U$ for some fixed $U$, then there exists some constant $C, c>0$ and $M_0>0$ such that for $M>M_0$ we have
\begin{equation}\label{eqn35}
|\mathcal{H}''(u)| \leq Ce^{-cu}e^{-cM}, \qquad |\mathcal{H}''_{\infty}(u)| \leq Ce^{-cu}e^{-cM}
\end{equation}
uniformly for $u>U.$
\end{lemma}

%
%
\begin{proof}
We have, by (\ref{eqn41})
\begin{eqnarray}
|Z_M\mathcal{H}''(u)|&\leq&\frac{\nu M^{1/3}}{\pi}\int_{\Gamma''}e^{\nu M^{1/3}u\Re(z-q)}\cdot e^{M\Re(f(z)-f(p))}\cdot\biggl|\frac{g(p)}{g(z)}\biggl|\cdot\frac{1}{|\nu M^{1/3}(z-p)|^k}dz \nonumber\\
& \leq & \frac{\nu M^{1/3}}{\pi(\nu M^{1/3}\delta)^k}\int_{\Gamma''}e^{\nu M^{1/3}u\Re(z-q)}\cdot e^{-\nu^3\delta^3M/6}\cdot\biggl|\frac{g(p)}{g(z)}\biggl|dz \label{eqn42}
\end{eqnarray}
We note that $|g(z)|$ is bounded above and below for $z\in\Gamma''$. Moreover since \\$\Re(z-q)\leq-\epsilon/(2\nu M^{1/3})$, we have
\begin{eqnarray*}
\exp\Big[\nu M^{1/3}u\Re(z-q)\Big] & \leq & C\exp\Big[\nu M^{1/3}(u-U)\Re(z-q)\Big] \\
& \leq & C\exp\Big[-\epsilon(u-U)/2\Big] = C\exp(-cu)
\end{eqnarray*}
for some constants $C$ and $c$.  Hence from (\ref{eqn42}) we have 
\[
|Z_M\mathcal{H}''(u)| \leq CM^{-(k-1)/3}e^{-cu}e^{-\nu^3\delta^3M/6} \leq Ce^{-cu}e^{-cM}
\]
for some constant $c, C>0$ and for $M$ large enough. The statement for $\mathcal{H}_{\infty}''(u)$ is proved in \cite{1}. 
\end{proof}
%
%

Now in order to prove (\ref{eqn31}) in Proposition \ref{prop2}, we just need to prove 
\[
|\mathcal{H}'(u) - \mathcal{H}'_{\infty}(u)|\leq Ce^{-cu}M^{-1/3} 
\]
for some constant $c, C>0$ and for $M$ large enough. First using change of variables $a\mapsto \nu M^{1/3}(z-p)$ in (\ref{eqn17}) for $\mathcal{H}'_\infty(u)$, we know
\begin{equation}\label{eqn43}
\mathcal{H}'_{\infty}  =  \frac{\nu M^{1/3}}{2\pi}\int_{\Gamma'}e^{\nu M^{1/3}u(z-q)}\cdot e^{-M\nu^3(z-p)^3/3}\cdot\frac{1}{[\nu M^{1/3}(z-p)]^k}dz.
\end{equation}
Let $\Gamma' = {\Gamma}'_0\cup{\Gamma}_1'$ where $\Gamma_0' := \Gamma'\cap\Gamma_0$ and $\Gamma_1' := \Gamma'\cap\Gamma_1$. We can define $\mathcal{H}'_{\infty, 0}, \mathcal{H}'_{\infty, 1}$ to be the integral in (\ref{eqn43}) on $\Gamma_0'\cup\overline{\Gamma_0'}$ and $\Gamma_1'\cup\overline{\Gamma_1'}$. Similarly we can define $\mathcal{H}_0'$ and $\mathcal{H}_1'$. The following lemma completes the proof of the first part of Proposition \ref{prop2}. 
%
%
\begin{lemma}\label{lemma3}
For some constant $c, C>0$ and for $M$ large enough, we have
\begin{equation}\label{eqn44}
|Z_M\mathcal{H}'_0(u) - \mathcal{H}'_{\infty, 0}(u)|\leq Ce^{-cu}M^{-1/3}, \quad |Z_M\mathcal{H}'_1(u) - \mathcal{H}'_{\infty, 1}(u)|\leq Ce^{-cu}M^{-1/3}  
\end{equation}
\end{lemma}
%
%
\begin{proof}
Assume $M$ to be large enough so that $\Gamma_0' = \Gamma_0$. For $z\in\Gamma_0$, $\nu M^{1/3}|z-p| = \epsilon/2.$ By (\ref{eqn38}) we have that $M(f(z)-f(p))$ is bounded. Also
\begin{eqnarray}
|e^{M(f(z)-f(p))} - e^{-M\nu^3(z-p)^3/3}| &\leq & CM|f(z)-f(p)-\nu^3(z-p)^3/3| \nonumber\\
& \leq & CM|z-p|^4 = CM^{-1/3}. \label{eqn46}
\end{eqnarray}
Also because $g(z)$ has no singularities or zeros around $z = p$, 
\[
\biggl|\frac{g(p)}{g(z)}-1\biggl| \leq C|z-p| = CM^{-1/3}.
\]
Thus 
\begin{eqnarray*}
&&\biggl|e^{M(f(z)-f(p))}\frac{g(p)}{g(z)} - e^{-M\nu^3(z-p)^3/3}\biggl| \\
&\leq & \biggl|e^{M(f(z)-f(p))} - e^{-M\nu^3(z-p)^3/3}\biggl| + \biggl|e^{M(f(z)-f(p))}\biggl|\cdot\biggl|\frac{g(p)}{g(z)}-1\biggl| \leq  CM^{-1/3}.
\end{eqnarray*}
Since we have
\begin{multline}
|Z_M\mathcal{H}'_0(u)-\mathcal{H}'_{\infty, 0}(u)|\leq \int_{\Gamma'_0} \frac{\nu M^{1/3}}{\pi}e^{\nu M^{1/3}u\Re(z-q)}\cdot \biggl|e^{M[f(z)-f(p)]}\frac{g(p)}{g(z)} - e^{-M\nu^3(z-p)^3/3}\biggl|\\
\cdot\frac{1}{|\nu M^{1/3}(z-p)|^k}dz, \label{eqn45}
\end{multline}
the integrand in (\ref{eqn45}) is bounded above by $Ce^{-cu}$. But the integral region is $\mathcal{O}(M^{-1/3})$. Hence the first part is (\ref{eqn44}) is true.

For the second part,  we have $z = p + te^{2\pi i/3}$ for $\epsilon/(2\nu M^{1/3})\leq t\leq\delta$. By (\ref{eqn38}), $\Re(f(z)-f(p))\leq-\nu^3t^3/6$. Like (\ref{eqn46}) we have
\[
|e^{M(f(z)-f(p))} - e^{-M\nu^3(z-p)^3/3}|\leq Ce^{-\nu^3t^3/6}M|z-p|^4 = CMe^{-cMt^3}t^4.
\]
Also
\[
\biggl|\frac{g(p)}{g(z)}-1\biggl| \leq C|z-p| = Ct.
\]
From these two inequalities we obtain
\[
\biggl|e^{M(f(z)-f(p))}\frac{g(p)}{g(z)} - e^{-M\nu^3(z-p)^3/3}\biggl|\leq Ce^{-cMt^3}(Mt^4+t).
\]
Hence, similar in (\ref{eqn45})
\begin{eqnarray*}
|Z_M\mathcal{H}'(u)-\mathcal{H}'_{\infty, 1}(u)| &\leq & \int_{\epsilon/(2\nu M^{1/3})}^\infty CM^{1/3}e^{-cu}\cdot e^{-cMt^3}(Mt^4+t)
\cdot\frac{1}{(M^{1/3}t)^k}dt \\
& \leq & M^{-1/3}\int_\epsilon^\infty Ce^{-cu}\cdot e^{-ct^3}\frac{1}{t^{k-1}}dt = Ce^{-cu}M^{-1/3}
\end{eqnarray*}
where in the second step we use the change of variables $t\mapsto t/(2\nu M^{1/3})$. 
\end{proof}
%
%

%
%
\subsection{Analysis of $\mathcal{J}(v)$}
This time we use the saddle point analysis for $\mathcal{J}(v)$ to prove (\ref{eqn32}) in Proposition \ref{prop2}. Since most of the analysis will be similar to that of the previous part, we just outline the key steps. 

The contour of $\Sigma$ in $\mathcal{J}(v)$ is graphed in Figure \ref{fig3}. Here $\Sigma = \bigcup_{i=0}^2\Sigma_{i}\cup\overline{\bigcup_{i=0}^2\Sigma_{i}}$ where 
\begin{eqnarray*}
\Sigma_{0} & = & \biggl\{\omega = p + \frac{3\epsilon}{\nu M^{1/3}} e^{i\pi t}: 0\leq t\leq \pi/3\biggl\}, \\
\Sigma_{1} & = & \biggl\{\omega = p+te^{i\pi/3}: \frac{3\epsilon}{2\nu M^{1/3}}\leq t\leq 1\biggl\}, \\
\Sigma_{2} & = & \biggl\{\omega = p + \frac{1}{2} + it: t\geq \frac{\sqrt{3}}{2}\biggl\}.
\end{eqnarray*}

\begin{figure}[!h]
    \centering
    \includegraphics[scale=.7]{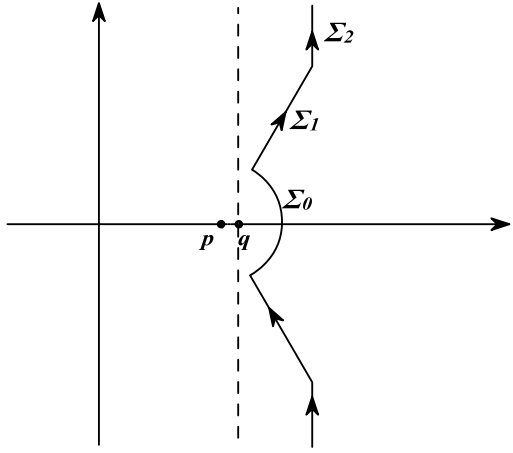}
    \caption{The Contour $\Sigma$ in $\mathcal{J}(v)$.}
    \label{fig3}
\end{figure}

As in the previous section, we choose $\Sigma$ in such a way that the following lemma holds. 

%
%
\begin{lemma}\label{lemma4}
$\Re(f(\omega))$ is monotonously increasing as $\omega$ travels along $\Sigma_{1}\sqcup\Sigma_{2}$. Moreover, for $z = p+1/2+it \in\Sigma_{2}$ (here $t\geq\sqrt{3}/2$), we have 
\begin{equation}\label{eqn49}
\Re(f(\omega) - f(\omega^*)) \geq C\log\biggl(\frac{t}{\sqrt{3}/2}\biggl)
\end{equation}
for some constant $C>0$, where $\omega^*$ is the interception of $\Sigma_{1}$ and $\Sigma_{2}$. 
\end{lemma}

%
%
\begin{proof}
\begin{itemize}
\item
For $\omega\in\Sigma_{1}$, $\omega = p+te^{i\pi/3}$ for some $3\epsilon(2\nu M^{1/3})^{-1}\leq t\leq 1$. Then
\[
\frac{d}{dt}\Re(f(\omega)) = -\frac{t^2\Big(t^2-(2p-1)t-2p^2+2p\Big)}{2p^2\Big(t^2+pt+p^2\Big)\Big(t^2+(p-1)t+(p-1)^2\Big)}.
\] 
It is straight forward to verify that the right hand side is positive for $t\in[0, 1]$. 
\item
For $\omega\in\Sigma_{2}$, $\omega = p+1/2+it$. Then
\begin{eqnarray*}
\frac{d}{dt}\Re(f(\omega)) = \frac{t\Big((2p-1)t^2+p(p-1)+(2p-1)/4\Big)}{p^2\Big(t^2+(p+1/2)^2\Big)\Big(t^2+(p-1/2)^2\Big)}>\frac{C}{t}>0.
\end{eqnarray*}
for some constant $C>0$, uniformly for all $t\geq{\sqrt{3}}/{2}$. 
Hence we have 
\[
\Re(f(w)-f(w^*))\geq C\log\biggl(\frac{t}{\sqrt{3}/2}\biggl). 
\]
\end{itemize}
\end{proof}
%
%

As in the last subsection, we can also define $\delta$ to satisfy (\ref{eqn37}). Similarly we can also define 
\[
\Sigma' = \{\omega\in\Sigma: |\omega-p|<\delta\}, \qquad\Sigma'' = \Sigma\backslash\Sigma',
\]
and $\Sigma_{\infty} = \Sigma_{\infty}'\sqcup\Sigma_{\infty}''$ where $\Sigma_{\infty}' := \nu M^{1/3}(\Sigma'-p)$ is the image of $\Sigma'$ under the map $\omega\mapsto \nu M^{1/3}(\omega-p)$, and $\Sigma_{\infty}'' = \Sigma_{\infty}\backslash\Sigma'_{\infty}.$ 

This time, for $\omega\in\Sigma_{1}\cap\Sigma'$ we have 
\begin{equation}\label{eqn46.1}
\Re(f(p + te^{i\pi/3}) - f(p)) \geq \frac{\nu^3}{6}t^3,
\end{equation}
and, because of Lemma \ref{lemma4} holds, we have for $\omega\in\Sigma''$,
\begin{equation}\label{eqn47}
\Re(f(\omega))-f(p)\geq \frac{\nu^3}{6}\delta^3.
\end{equation}
Specifically, for $\omega\in\Sigma_{2}$, by Lemma \ref{lemma4},
\begin{equation}\label{eqn50}
\Re(f(\omega)-f(p)) =\Re(f(\omega)-f(w^*)) + \Re(f(w^*)-f(p)) \geq \frac{\nu^3}{6}\delta^3 + C\log\biggl(\frac{t}{\sqrt{3}/2}\biggl)
\end{equation}
for some constant $C>0$. 

Similarly, from the decomposition of $\Sigma = \Sigma'\sqcup\Sigma''$ and $\Sigma_{\infty} = \Sigma_{\infty}'\sqcup\Sigma_{\infty}''$ we can also decompose $\mathcal{J}(v)$ and $\mathcal{J}_{\infty}(v)$ into
\[
\mathcal{J}(v) = \mathcal{J}'(v)+ \mathcal{J}''(v), \qquad \mathcal{J}_\infty(v) = \mathcal{J}'_\infty(v)+ \mathcal{J}''_\infty(v).
\]

Because of (\ref{eqn47}) and (\ref{eqn50}), the following lemma holds true. Since the proof is similar to that of Lemma \ref{lemma2}, we just provide the outlines. The only different part is that now our $\Sigma$ is an infinite contour.
%
%
\begin{lemma}\label{lemma5}
If $v$ is bounded below, i.e., $v>V$ for some fixed $V$, then there exists some constant $C, c>0$ and $M_0>0$ such that for $M>M_0$ we have 
\begin{equation}\label{eqn48}
|Z_M^{-1}\mathcal{J}''(v)|\leq Ce^{-cv}e^{-cM}, \qquad |\mathcal{J}_{\infty}''(v)|\leq Ce^{-cv}e^{-cM}
\end{equation}
uniformly for $v>V$. 
\end{lemma}

%
%
\begin{proof}
The result for $\mathcal{J}_{\infty}''(v)$ is proved in \cite{1}. We just provide a brief proof about $Z_M^{-1}\mathcal{J}''(v)$. The only difference from the previous subsection is that now $\Sigma''$ is not of finite length.  First since $\Re(\omega-q)\leq\epsilon/(2\nu M^{1/3})$, we obtain, similar to Lemma \ref{lemma2}, that  
\[
\exp\Big[-\nu M^{1/3}u\Re(\omega-q)\Big] \leq C\exp(-cv). 
\]
for some $C, c$, uniformly for $v>V$. We decompose $\mathcal{J}''(v)$ into $\mathcal{J}''(v)  = \mathcal{J}''_1(v) + \mathcal{J}''_2(v)$ where $\mathcal{J}''_1(v)$, $\mathcal{J}_2''(v)$ are the integral of (\ref{eqn22}) on $\Sigma''\cap\Sigma_{1}$ and $\Sigma''\cap\Sigma_{2}$. Since $\Sigma''\cap\Sigma_{1}$ is of finite length, we can take the same procedure as in the previous subsection to  prove that 
\[
|Z^{-1}_M\mathcal{J}_1''(v)|\leq Ce^{-cv}e^{-cM}. 
\]
Now we consider $\mathcal{J}''_2(v)$. For $\omega\in\Sigma_{2}''$, by using (\ref{eqn50}) we obtain 
\begin{eqnarray*}
|Z_M^{-1}\mathcal{J}''(v)|&\leq&\frac{\nu M^{1/3}}{\pi}\int_{\Sigma_{2}''}e^{-\nu M^{1/3}u\Re(\omega-q)} e^{-M\Re(f(\omega)-f(z))}\biggl|\frac{g(\omega)}{g(p)}\biggl||\nu M^{1/3}(\omega-p)|^kd\omega \\
& \leq & CM^{(r+1)/3}\int_{\sqrt{3}/2}^\infty e^{-cu}e^{-cM-cM\log (2t/\sqrt{3})}t^rdt  \\
& \leq & Ce^{-cM}e^{-cu}. 
\end{eqnarray*}
\end{proof}
%
%

After showing that $\mathcal{J}''(v)$ is negligible, we then turn to analyze $\mathcal{J}'(v)$, we can decompose $\Sigma' = \Sigma_{0}'\sqcup\Sigma_{1}'$ where $\Sigma_{0}' = \Sigma'\cap\Sigma_{0}, \Sigma_{1}' = \Sigma'\cap\Sigma_{1}$. Also we can use the same rule to write $\Sigma_{\infty}' = \Sigma_{\infty, 0}'\sqcup\Sigma_{\infty, 1}'$.
From this we can similarly decompose 
\[
\mathcal{J}'(v) = \mathcal{J}'_{0}(v)+\mathcal{J}'_{1}(v), \qquad  \mathcal{J}'_{\infty}(v) = \mathcal{J}'_{\infty, 0}(v)+\mathcal{J}'_{\infty, 1}(v)
\]

As a analog of Lemma \ref{lemma3}, we have the following lemma. The proof is similar and thus we just omit that.
 
%
%
\begin{lemma}\label{lemma6}
For some constant $c, C>0$ and for $M$ large enough, we have
\begin{equation}\label{eqn51}
|Z_M^{-1}\mathcal{J}'_{0}(v)-\mathcal{J}'_{\infty, 0}(v)|<Ce^{-cv}M^{-1/3}, \quad |Z_M^{-1}\mathcal{J}'_{1}(v)-\mathcal{J}'_{\infty, 1}(v)|<Ce^{-cv}M^{-1/3}
\end{equation}
uniformly for $v$ bounded below. 
\end{lemma}

As a final remark, Lemma \ref{lemma5} and Lemma \ref{lemma6} together implies the second part of Proposition \ref{prop2}.

%
%
\subsection{Finishing the proof}
By (\ref{eqn6}) in Corollary \ref{corollary1} and by the scaling in (\ref{eqn11}) we obtain
\begin{equation}\label{eqn52}
\mathbb{P}\biggl\{\frac{M^{2/3}(\gamma-1)^{4/3}}{\gamma}\biggl[\Big(1-\frac{1}{\gamma}\Big)^2-\lambda_{\min}\biggl]\leq x\biggl\} \to \det(I-\mathcal{K}_{\infty})
\end{equation}
where $\mathcal{K}_{\infty}$ is an integral operator on $L^2(0, \infty)$ with kernel
\[
\mathcal{K}_{\infty}(u, v) = -\int_0^\infty\mathcal{H}_{\infty}(y+x+u)\mathcal{J}_{\infty}(y+x+v)dy. 
\]
for $\mathcal{H}_\infty(u), \mathcal{J}_\infty(v)$ defined in (\ref{eqn29}) and (\ref{eqn30}). By the same argument  in section 3.3 of \cite{1}, we can prove that the Fredholm determinant in (\ref{eqn52}) is just $F_k(x)$. This finishes the proof.

\section{Proof of Part 2 of Theorem \ref{thm1}}\label{sec5}
In this part, we analyze the limiting distribution of the smallest sample eigenvalue when there are some ``true'' eigenvalues being smaller than the critical point $1-\gamma^{-1}$. 

We assume that for some $1\leq k\leq r_2$, 
\[
\pi_N^{-1} = \pi^{-1}_{N-1} = \ldots = \pi^{-1}_{N-k+1} < 1-\gamma^{-1}. 
\]
We also assume that the rest of the $\pi^{-1}$'s are in a compact subset of $(\pi_N^{-1}, \infty)$. In this case, our constants are 
\begin{equation}\label{eqn53}
\alpha = \frac{1}{2}, \quad \mu = \frac{1}{\pi_N}-\frac{\gamma^{-2}}{\pi_N-1}, \quad \nu = \sqrt{\frac{1}{\pi_N^2}-\frac{1}{\gamma^2(\pi_N-1)^2}}, \quad q = \pi_N+\frac{\epsilon}{\nu M^{1/2}}.
\end{equation}
Here $\epsilon>0$ is a pre-fixed smaller number. We note here that due to some very simple algebra we can verify that $\nu$ is a real number. Again, we define 
\begin{equation}\label{eqn54}
g(z) = \frac{1}{(z-1)^{r_1+r_2-k}}\prod_{\ell\leq r_2\text{ or }N-r_1+1\leq\ell\leq N-k}(z-\pi_\ell). 
\end{equation}
The function $\mathcal{H}(u), \mathcal{J}(v)$ is now
\begin{eqnarray}
\mathcal{H}(u) & = & \frac{\nu M^{1/2}}{2\pi}\int_\Gamma dz e^{\nu M^{1/2}u(z-q)}\cdot e^{Mf(z)}\cdot\frac{1}{g(z)}\cdot\frac{1}{(z-\pi_N)^k}dz\label{eqn55}\\
\mathcal{J}(u) & = & \frac{\nu M^{1/2}}{2\pi}\int_{\Sigma} d\omega e^{-\nu M^{1/2}u(\omega-q)}\cdot e^{-Mf(\omega)}\cdot g(\omega){(\omega-\pi_N)^k}d\omega\label{eqn56}
\end{eqnarray}
The constant $Z_M$ is defined to be
\begin{equation}\label{eqn57}
Z_M = \frac{g(\pi_N)}{e^{Mf(\pi_N)}(\nu M^{1/2})^k}.
\end{equation}
Finally, like the previous section, we define our $\mathcal{H}_\infty(u)$ and $\mathcal{J}_\infty(v)$ to be 
\begin{eqnarray}
\mathcal{H}_\infty(u) & = & e^{-\epsilon u}\mathrm{Res}_{a=0}\biggl\{e^{ua-a^2/2}\cdot\frac{1}{a^k}\biggl\}, \label{eqn58}\\
\mathcal{J}_\infty(v) & = & \frac{1}{2\pi}e^{\epsilon v}\int_{\Sigma_\infty}s^ke^{s^2/2-vs}ds \label{eqn59}. 
\end{eqnarray}
where 
\[
\Sigma_\infty = \{z = 2\epsilon+it : t\in\mathbb{R}\}
\]
oriented from bottom to top. We note that the integral on $\Sigma_\infty$ in (\ref{eqn59}) also equals to the integral on the imaginary axis, from bottom to top. 
Again, in this section, our main goal is to prove the following proposition.

%
%
\begin{proposition}\label{prop3}
Assume $u>U, v>V$ are bounded below, then
\begin{enumerate}
\item
There exists some constant $C, c>0$, $M_0>0$ such that uniformly for $u>U$
\begin{equation}\label{eqn60}
|Z_M\mathcal{H}(u)-\mathcal{H}_\infty(u)| \leq Ce^{-cu}M^{-1/2}, \qquad \forall M>M_0.
\end{equation}
\item
There exists some constant $C, c>0$, $M_0>0$ such that uniformly for $v>V$
\begin{equation}\label{eqn61}
|Z_M^{-1}\mathcal{J}(v)-\mathcal{J}_\infty(v)| \leq Ce^{-cv}M^{-1/2}, \qquad \forall M>M_0.
\end{equation}
\end{enumerate}
\end{proposition}

In order to prove this proposition, we also need the saddle point analysis on $f(z)$. Note that with the choice of $\mu$ in (\ref{eqn53}), the two saddle points are $z = \pi_N$ and $z = 1/(\mu\pi_N)$. It is not hard to know
\begin{equation}\label{eqn62}
1<\frac{1}{\mu\pi_N}<\frac{\gamma}{\gamma-1}<\pi_N. 
\end{equation}
and
\begin{equation}\label{eqn63}
f'(\pi_N) = 0, \qquad f''(\pi_N) = -\nu^2. 
\end{equation}
In subsection \ref{sec3.2.1} we use the residue theorem to provide an asymptotic analysis for $\mathcal{H}(u)$ in (\ref{eqn60}), and in subsection \ref{sec3.2.2}, a saddle point analysis around $\pi_N$ will provide us a good approximation for $\mathcal{J}(v)$ in (\ref{eqn61}). Finally, in subsection \ref{sec3.2.3} we finish the proof.

%
%
\subsection{Analysis of $\mathcal{H}(u)$}\label{sec3.2.1}
By the residue theorem we have
\begin{equation}\label{eqn64}
Z_M\mathcal{H}(u)  =  \mathcal{H}_1(u) + \frac{\nu M^{1/2}}{2\pi}\int_{\Gamma'} e^{\nu M^{1/2}u(z-q)}\cdot e^{M(f(z) - f(\pi_N))}\cdot\frac{g(\pi_N)}{g(z)}\cdot\frac{1}{[\nu M^{1/2}(z-\pi_N)]^k}dz
\end{equation}
where 
\begin{equation}\label{eqn65}
\mathcal{H}_1(u) := ie^{-\epsilon u}\mathrm{Res}_{a=0}\biggl\{\frac{e^{au}}{a^k}\cdot e^{M[f(\pi_N+a/(\nu\sqrt{M}))-f(\pi_N)]}\cdot\frac{g(\pi_N)}{g(\pi_N+a/(\nu\sqrt{M}))}\biggl\}
\end{equation}
and $\Gamma'$ is a contour inclosing $\pi_1, \ldots, \pi_{N-k}$ but excluding $\pi_{N-k+1} = \ldots = \pi_N$. We will later choose $\Gamma'$ explicitly, based on which we will prove that the integral on $\Gamma'$ is negligible. Hence the main contribution of (\ref{eqn64}) is the residue $\mathcal{H}_1(u)$ in (\ref{eqn65}). 

The following lemma gives an approximation of $\mathcal{H}_1(u)$.
%
%
\begin{lemma}\label{lemma7}
For $u>U$ bounded below, uniformly there exists constants $c, C>0$ and $M_0>0$ such that for all  $ M>M_0$
\[
|\mathcal{H}_1(u) -\mathcal{H}_\infty(u)|\leq Ce^{-cu}M^{-1/2}. 
\]
\end{lemma}

%
%
\begin{proof}
The proof is almost the same as that in \cite{1}. \\First by expanding $f(\pi_N-a/(\nu\sqrt{M}))-f(\pi_N)$ and on  $g(\pi_N)/g(\pi_N+a/(\nu\sqrt{M}))$ around $a=0$ we obtain 
\[
\frac{e^{ua}}{a^k}\cdot e^{M[f(\pi_N+a/(\nu\sqrt{M}))-f(\pi_N)]}\cdot\frac{g(\pi_N)}{g(\pi_N+a/(\nu\sqrt{M}))} = \frac{e^{ua}}{a^k}e^{-a^2/2}\biggl[1+\frac{1}{\sqrt{M}}\phi_M(a)\biggl]
\]
where $\phi_M(a)$ is a analytic function around $a= 0$. Moreover all the coefficient in the Taylor expansion of $\phi_M(a)$ around $a=0$ are bounded in $M$. Hence we obtain 
\begin{equation}\label{eqn66}
\mathcal{H}(u) = \mathcal{H}_\infty(u) + \frac{e^{-\epsilon u}}{\sqrt{M}}Q_M(u). 
\end{equation}
Here $Q_M(u)$ is a polynomial of $u$ of degree at most $k-1$, and all the coefficients of $Q_M(u)$ are bounded in $M$. Due to the factor $e^{-\epsilon u}$ in (\ref{eqn66}) we complete the proof. 
\end{proof}
%
%

To keep our promise our next step is to prove that the second term in (\ref{eqn64}) is negligible. This is done in Lemma \ref{lemma9}. But before that, we need to choose our contour $\Gamma'$ explicitly. Again, our contour $\Gamma'$ is symmetric with respect to the real axis thus we just  define the upper part. We take $\Gamma := \bigcup_{i=1}^6\Gamma_i\cup\overline{\bigcup_{i=1}^6\Gamma_i}$ where the $\Gamma_i$'s will be defined below. For a graph of the contour please refer to Figure \ref{fig4}. 
\begin{figure}[!h]
    \centering
    \includegraphics[scale=.7]{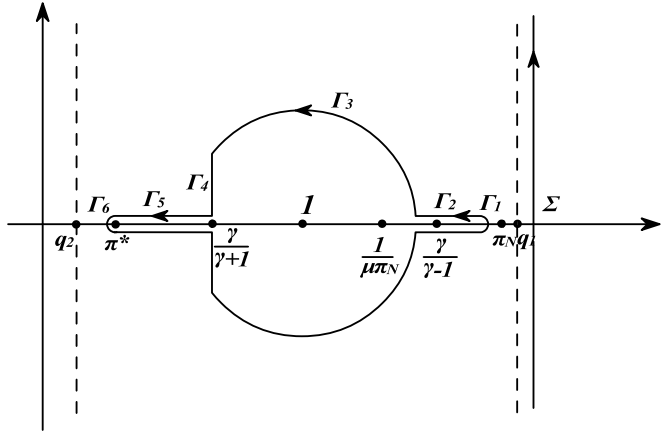}
    \caption{The Contours $\Gamma$ and $\Sigma$}
    \label{fig4}
\end{figure}
\begin{itemize}
\item
Take $\pi'$ such that $\max(\pi_{N-k}, \gamma/(\gamma-1)) < \pi' < \pi_N$. Moreover take some $\kappa>0$ to be a sufficiently small constant. We take $\Gamma_1$ to be the quarter-circle centered at $\pi'$, with radius $\kappa$. Here $\kappa>0$ is a small constant.   
\item
Take $R = 1/\sqrt{\mu}-1$. We define
\[
\Gamma_2 = \{z = t + i\kappa : 1+R<t<\pi'\}. 
\]

\item
Define $\Gamma_3$ to be the arc centered at $1$ with radius $R := 1/\sqrt{\mu}-1$. Using some simple algebra, we can verify that with our choice of $R$, the real part of the intersection of $\Gamma_2$ and $\Gamma_3$ is between  $1/(\mu\pi_N)$ and $\gamma/(\gamma-1)$, provided that $\kappa$ is small enough. 
\item
If $1-R<\gamma/(\gamma+1)$, we just define $\Gamma_4 = \{\gamma/(1+\gamma)+it : t\in\mathbb{R}\}$ to be the vertical line joining $\Gamma_3$ and $\Gamma_5$. This is shown in Figure \ref{fig4}. If $1-R\geq\gamma/(\gamma+1)$, then we just exclude $\Gamma_4$ in the definition of $\Gamma' = \bigcup_{i=1}^6\Gamma_i\cup\overline{\bigcup_{i=1}^6\Gamma_i}$, connecting $\Gamma_3$ and $\Gamma_5$ directly as we did in the previous section. 
\item
Take $\pi^* = \min(\pi_1, \gamma/(\gamma+1))$, we define
\[
\Gamma_5 = \{z = t + i\kappa : \pi^*<t<\max(1-R, \gamma/(1+\gamma))\}. 
\]
\item
We finally define $\Gamma_6$ to be the quarter-circle centered at $\pi^*$ with radius $\kappa$. 
\end{itemize}

Our analysis of $\mathcal{H}(u)$ replies heavily on the following lemma. 
%
%
\begin{lemma}\label{lemma8}
There exists some constant $c>0$ such that 
\[
\sup_{z\in\Gamma'}\Re\Big[f(z)-f(\pi_N)\Big] \leq -c. 
\]
\end{lemma}

%
%
\begin{proof}
Similar to Lemma \ref{lemma1}, we also divide the proof into several parts.
\begin{itemize}
\item
In $\Gamma_2$, $z$ can be represented as $z = t+i\kappa$. We note that when $\kappa = 0$, \\$z = t\in(1/{\mu\pi_N}, \pi_N)$ and
\[
\frac{d}{dt}\Re(f(z)) = \frac{\mu(\pi_N-t)(t-1/(\mu\pi_N))}{t(t-1)} > 0. 
\]
Thus $\Re(f(z))$ is strictly decreasing for $t$ traveling on the real axis $\pi_N$ to $1/(\mu\pi_N)$. Moreover $\Re(f(t+i\kappa))$ is continuous for $\kappa$ around zero, uniformly for all $t$ lying in a compact subset of $(1/(\mu\pi_N), \pi_N)$. Thus we can choose $\kappa$ to be sufficiently small such that for some constant $c>0$,
\[
\sup_{z\in\Gamma_1\cup\Gamma_2}\Re\Big[f(z)-f(\pi_N)\Big] \leq -c. 
\]
\item
In $\Gamma_3$, $z = 1+Re^{i\theta}$ for $\theta$ in a subset of $(0, \pi)$. Then
\begin{eqnarray*}
\frac{d}{d\theta}\Re(f(z)) & = & \frac{R\sin\theta}{1+2R\cos\theta+R^2}\Big[\mu(1+R^2+R\cos\theta)-1\Big] \\
& \leq & \frac{R\sin\theta}{1+2R\cos\theta+R^2}\Big[\mu(1+R)^2-1\Big] = 0
\end{eqnarray*}
with our choice of $R := 1/\sqrt{\mu}-1$. Hence $\Re(f(z))$ is decreasing in $\Gamma_3$. 
\item
In $\Gamma_4$(if it exists), $z = \gamma/(1+\gamma) + it$. With exactly the same calculation as that of Lemma \ref{lemma1}, we can prove that $\Re(f(z))$ is decreasing in $\Gamma_4$.
\item
in $\Gamma_5$, $z$ can be represented as $z = t+i\kappa$, just as that in $\Gamma_2$. This time when $\kappa= 0$, $z = t \in (\pi^*, \gamma/(1+\gamma))$ and 
\[
\frac{d}{dt}\Re(f(z)) = \frac{\mu(\pi_N-t)(1/(\mu\pi_N)-t)}{t(1-t)} > 0. 
\]
With the same argument as that in $\Gamma_2$, we can choose $\kappa$ sufficiently small such that 
\[
\sup_{z\in\Gamma_5\cup\Gamma_6}\Re\Big[f(z)-f(\pi_N)\Big] \leq -c. 
\]
\end{itemize}
\end{proof}
%
%

Based on Lemma \ref{lemma8}, our next lemma establishes that the integral along $\Gamma'$ in (\ref{eqn64}) is negligible, which finishes the proof of (\ref{eqn60}) in Proposition \ref{prop3}. 

%
%
\begin{lemma}\label{lemma9}
There exists some $C, c>0$ and $M_0>0$ such that uniformly for $u>U$ and for $M>M_0$
\[
\biggl|\frac{\nu M^{1/2}}{2\pi}\int_{\Gamma'} e^{\nu M^{1/2}u(z-q)}\cdot e^{M(f(z) - f(\pi_N))}\cdot\frac{g(\pi_N)}{g(z)}\cdot\frac{1}{[\nu M^{1/2}(z-\pi_N)]^k}dz\biggl|\leq Ce^{-cu}e^{-cM}.
\]
\end{lemma}

%
%
\begin{proof}
For $z\in\Gamma'$, by Lemma \ref{lemma8}, we have 
\[
e^{M\Re(f(z)-f(\pi_N))} \leq e^{-cM}. 
\]
Moreover, since $\nu M^{1/2}\Re(z-q)\leq-\epsilon$ for $z\in\Gamma'$, by the same proof as that in Lemma \ref{lemma2} we have $e^{\nu M^{1/2}u\Re(z-q)}\leq Ce^{-cu}$ for some constant $c, C>0$, uniformly for $u>U$ and for $M$ sufficiently large. One last notes that 
\[
\frac{g(\pi_N)}{g(z)}\cdot \frac{1}{[\nu M^{1/2}(z-\pi_N)]} \leq C
\]
remains bounded for $z\in\Gamma'$. This finishes the proof of Lemma \ref{lemma9}. 
\end{proof}
%
%

%
%
\subsection{Analysis of $\mathcal{J}(v)$}\label{sec3.2.2}
The contour for $\mathcal{J}(v)$ is fairly simple compared to $\Gamma.$ As is also shown in Figure \ref{fig4}, the contour $\Sigma$ is defined by 
\[
\Sigma = \biggl\{\omega = \pi_N + \frac{2\epsilon}{\nu M^{1/2}} + it : t\in\mathbb{R}^+\biggl\}. 
\]

Now if we have a constant $\delta > 0$ such that  
$0<\delta < ({\pi_N-1})/{2}$ 
then for $\omega$ such that $|\omega-\pi_N|<\delta$ we have
\begin{equation}\label{eqn69}
\frac{1}{3!}|f'''(\omega)|\leq \frac{1}{3} + \frac{8}{3\gamma^2(\pi_N-1)^3} := C_0.
\end{equation}
We then fix $\delta$ to satisfy
\begin{equation}\label{eqn68}
0<\delta <\min\biggl\{{\frac{\pi_N-1}{2}}, \frac{\nu^2}{4C_0}\biggl\}
\end{equation}
Then for $|\omega-\pi_N|<\delta$, 
\begin{eqnarray}
\biggl|f(\omega) - f(\pi_N)-\frac{f''(\pi_N)}{2}(\omega-\pi_N)^2\biggl| &\leq& \max_{|s-\pi_N|\leq\delta}\frac{|f'''(s)|}{3!}|\omega-\pi_N|^3 \nonumber\\
& \leq & C_0|\omega-\pi_N|^3 \leq \frac{\nu^2}{4}|\omega-\pi_N|^2.\label{eqn70}
\end{eqnarray}

Now we decompose $\Sigma = \Sigma' \sqcup\Sigma''$ where 
\[
\Sigma' = \{\omega\in\Sigma: |\omega-\pi_N|<\delta\}
\]
and $\Sigma'' = \Sigma\backslash\Sigma'$. We further define $\Sigma'_\infty := \nu M^{1/2}(\Sigma'-\pi_N)$ to be the image of $\Sigma'$ under the mapping. Moreover define $\Sigma''_\infty := \Sigma_\infty\backslash\Sigma_\infty'$. Based on the integration contour, we can also set
\begin{equation}\label{eqn71}
\mathcal{J}(v) = \mathcal{J}'(v) + \mathcal{J}''(v), \qquad \mathcal{J}_\infty(v) = \mathcal{J}'_\infty(v) + \mathcal{J}''_\infty(v)
\end{equation}
where the decomposition is similar to that of the previous subsection. 

Define $\omega^*$ as the intersection of $\Sigma'$ and $\Sigma''$. 
From (\ref{eqn70}) and from the parametrization $\omega^* = \pi_N + 2\epsilon M^{-1/2}/\nu + it^*$ for some $t^*$ bounded below we obtain
\begin{eqnarray}
\Re\Big[f(\omega^*) - f(\pi_N)\Big] &\geq & -\frac{\nu^2}{2}\Re(\omega^*-\pi_N)^2 -\frac{\nu^2}{4}|\omega^*-\pi_N|^2 \nonumber \\
& = & -\frac{3\epsilon^2}{M} + \frac{1}{4}\nu^2t^{*2}\geq \frac{\nu^2}{32}\delta^2\label{imp}
\end{eqnarray}
if $M$ is large enough such that $t^{*}\geq \delta/2$ and $3\epsilon^2/M\leq \nu^2\delta^2/32$.

To analyze the behavior of $\mathcal{J}(v)$, we have the following lemma. 
%
%
\begin{lemma}\label{lemma10}
$\Re(f(\omega))$ is increasing as $\omega$ travels along $\Sigma$. 
Moreover, for $\omega\in\Sigma''$ and for sufficiently large $M$, we have 
\begin{equation}\label{eqn72}
\Re(f(\omega) - f(\pi_N))\geq C_1 + C\log(t/\delta)
\end{equation}
for some constant $C_1, C>0$ and $t$ is defined by $\omega = \pi_N + {2\epsilon}/{(\nu M^{1/2})} + it$. 
\end{lemma}
%
%
\begin{proof}
The proof is straightforward as we observe for $\omega = \pi_N + {2\epsilon}/{(\nu M^{1/2})} + it$,
\begin{equation}\label{eqn67}
\frac{d}{dt}\Re(f(\omega)) = \frac{t\Big[(\gamma^2-1)t^2 + \gamma^2(A-1)^2-A^2\Big]}{\gamma^2(A^2+t^2)((A-1)^2+t^2)}
\end{equation}
where $A := \pi_N + 2\epsilon/(\nu M^{1/2})$. We note that $\gamma^2(A-1)^2-A^2 > 0$ is equivalent to that $A\geq \gamma/(\gamma-1)$, which is indeed true. Hence (\ref{eqn67}) is positive. 

For the second statement, if $t\geq \delta$, then from (\ref{eqn67}) 
\[
\frac{d}{dt}\Re(f(\omega)) \geq \frac{C}{t}
\]
for some constant $C$. Integrating this from $t^*$ to $t$, tonether with (\ref{imp}) gives the desired result. 
\end{proof}
%
%

For $f(\omega)$,  $\omega = \pi_N$ is the saddle point. So we expect that all the integral on $\Sigma''$ and $\Sigma''_\infty$ is negligible, as the following lemma states.
%
%
\begin{lemma}\label{lemma11}
If $v$ is bounded below, i.e., $v>V$ for some fixed $V$, then there exists some constant $C, c>0$ and $M_0>0$ such that for $M>M_0, v>V$,
\[
|Z_M^{-1}\mathcal{J}''(v)| \leq Ce^{-cv}e^{-cM}, \qquad |\mathcal{J}_\infty''(v)| \leq Ce^{-cv}e^{-cM}.
\]
\end{lemma}
%
%
\begin{proof}
The result for $\mathcal{J}_\infty''(v)$ is proved in \cite{1} thus we just prove the result for $Z_M^{-1}\mathcal{J}''(v)$. Recall that
\[
Z_M^{-1}\mathcal{J}''(v) = \frac{\nu M^{1/2}}{\pi}\int_{\Sigma''}d\omega e^{-\nu M^{1/2}v(\omega-q)}\cdot e^{-M(f(\omega)-f(\pi_N))}\cdot\frac{g(\omega)}{g(\pi_N)}\cdot\Big[\nu M^{1/2}(\omega-\pi_N)\Big]^k.
\]
Note that by Lemma \ref{lemma10} we have 
\[
e^{-M\Re(f(\omega)-f(\pi_N))} \leq e^{-cM}\biggl|\frac{t}{\delta}\biggl|^{-cM}. 
\]
where $t$ is the paralyzation of $\omega = \pi_N+2\epsilon M^{-1/2}/\nu + it$. 
Moreover $e^{-\nu M^{1/2}v\Re(\omega-q)} = e^{-\epsilon v}$. Finally $g(\omega)(\omega-\pi_N)^k$ is just of polynomial growth of $t$. We get that 
\begin{eqnarray*}
|Z_M^{-1}\mathcal{J}''(u)| & \leq & Ce^{-cM}e^{-cv}\int_{\delta}^\infty\biggl|\frac{t}{\delta}\biggl|^{-cM}|t|^{r+1/2}dt\leq Ce^{-cM}e^{-cv}
\end{eqnarray*}
for some constants $C, c>0$ and for $M$ sufficiently large. 
\end{proof}
%
%

As expected, our next step is to prove that $Z_M^{-1}\mathcal{J}'(v)$ and $\mathcal{J}'_\infty(v)$ are sufficiently close to each other. This is proved in Lemma \ref{lemma12}.
%
%
\begin{lemma}\label{lemma12}
If $v$ is bounded below, i.e., $v>V$ for some fixed $V$, then there exists some constant $C, c>0$ and $M_0>0$ such that for $M>M_0, v>V$,
\[
|Z_M^{-1}\mathcal{J}'(v)-\mathcal{J}'_\infty(v)| \leq Ce^{-cv}M^{-1/2}.
\]
\end{lemma}

%
%
\begin{proof}
By the change of variables $s \mapsto \nu M^{1/2}(\omega - \pi_N)$ in $\mathcal{J}'_\infty$ in (\ref{eqn59}) we have
\begin{multline}
|Z_M^{-1}\mathcal{J}'(v)-\mathcal{J}_\infty'(v)| \leq \\
\frac{\nu M^{1/2}e^{-\epsilon v}}{\pi}\int_{\Sigma'}d\omega \biggl|e^{-M(f(\omega)-f(\pi_N))}\frac{g(\omega)}{g(\pi_N)}-e^{\nu^2M(\omega-\pi_N)^2/2}\biggl|\cdot\Big[\nu M^{1/2}(\omega-\pi_N)\Big]^k.\label{eqn73}
\end{multline}
Now we follow the same proof as that in Lemma \ref{lemma3} thus we just list the key steps. We can prove 
\[
|e^{-M(f(\omega)-f(\pi_N))}-e^{\nu^2M(\omega-\pi_N)^2/2}| \leq  CM^{-1/2}e^{c\Re(s^2)}s^3
\]
where $s$ is defined by $\omega =  \pi_N+(\nu M^{1/2})^{-1}s$. Further $|g(\omega)/g(\pi_N)-1|\leq CM^{-1/2}s$. Thus, by the change of variables $\omega =  \pi_N+(\nu M^{1/2})^{-1}s$ in (\ref{eqn73}) we obtain
\[
|Z_M^{-1}\mathcal{J}'(v)-\mathcal{J}_\infty'(v)|\leq \frac{e^{-\epsilon v}}{\pi}\int_{2\epsilon-\infty i}^{2\epsilon+\infty i}CM^{-1/2}e^{c\Re(s^2)}(s+s^3)\cdot s^kds = Ce^{-\epsilon v}M^{-1/2}
\]
for some constant $c, C>0$. 
\end{proof}
%
%

%
%
\subsection{Finishing the proof}\label{sec3.2.3}
Using the same steps as in \cite{1}, with minor modifications, we can prove
\begin{equation}\label{eqn74}
\mathcal{H}_\infty(u) = e^{-\epsilon u}\cdot\frac{He_{k-1}(u)}{(k-1)!}, \qquad \mathcal{J}_\infty(v) = e^{\epsilon v}e^{-v^2/2}\cdot He_k(v).
\end{equation}
where $He_k(x)$ is the $k$-th probabilists' Hermite polynomial, that is, 
\[
He_k(x) := (-1)^ke^{x^2/2}\biggl(\frac{d}{dx}\biggl)^ke^{-x^2/2}.
\]
Thus the kernel $\mathcal{K}(u, v)$ has the desired form, as is proved by \cite{1}. We omit the steps.

\section{Independence of the extreme eigenvalues}\label{sec6}
In this section we will prove that the largest eigenvalue and the smallest eigenvalue are asymptotically independent. In Section \ref{sec4} and Section \ref{sec5} we dropped the subscript one for notational convenience. In this section, however, we will consider both the maximum and the minimum eigenvalue. Hence we need these subscripts to distinguish them.

Note that we have four cases here.
\begin{itemize}
\item
Case 1. $\ell_1 = \pi_1^{-1}\leq 1+\gamma^{-1}$ and $\ell_N = \pi_N^{-1}\geq 1-\gamma^{-1}$.
\item
Case 2. $\ell_1 = \pi_1^{-1}> 1+\gamma^{-1}$ and $\ell_N = \pi_N^{-1}\geq 1-\gamma^{-1}$.
\item
Case 3. $\ell_1 = \pi_1^{-1}\leq 1+\gamma^{-1}$ and $\ell_N = \pi_N^{-1}< 1-\gamma^{-1}$.
\item
Case 4. $\ell_1 = \pi_1^{-1}> 1+\gamma^{-1}$ and $\ell_N = \pi_N^{-1}< 1-\gamma^{-1}$.
\end{itemize}
Since the proofs for all the cases are almost exactly the same, thus here we just prove the theorem for case 1. 

Recall that in Proposition \ref{prop1}, we proved that 
\begin{equation}\label{t2}
\mathbb{P}(\xi_1\leq\lambda_{\min}\leq\lambda_{\max}\leq\xi_2) = \det\biggl(I-
\left(
\begin{array}{cc}
{K}_{11} & {K}_{12} \\
{K}_{21} & {K}_{22} \\
\end{array}
\right)
\biggl). 
\end{equation}
where the operators $K_{ji}$ are defined in (\ref{eqn5}). Our first observation is that for any constant  $W_M$, we will have
\begin{eqnarray}
\mathbb{P}(\xi_1\leq\lambda_{\min}\leq\lambda_{\max}\leq\xi_2)& =& \det\biggl(I-
\left(
\begin{array}{cc}
I & 0 \\
0 & W_MI \\
\end{array}
\right)
\left(
\begin{array}{cc}
{K}_{11} & {K}_{12} \\
{K}_{21} & {K}_{22} \\
\end{array}
\right)
\left(
\begin{array}{cc}
I & 0 \\
0 & W_M^{-1}I \\
\end{array}
\right)
\biggl) \nonumber \\
& = &  
\det\biggl(I-
\left(
\begin{array}{cc}
{K}_{11} & W_M^{-1}{K}_{12} \\
W_M{K}_{21} & {K}_{22} \\
\end{array}
\right)
\biggl). \label{t1}.
\end{eqnarray}
We will leave the constant $W_M$ as an normalization constant, to be defined later. 

As the next step, to ensure the probability to be non-trivial, we need to take the proper scaling of  $\xi_1$ and $\xi_2$. 
Under the case 1, define 
\begin{equation}\label{eqn6.1}
\xi_1 = \mu_1 - \frac{\nu_1}{M^{\alpha_1}}x, \qquad\xi_2 = \mu_2 + \frac{\nu_2}{M^{\alpha_2}}y. 
\end{equation}
Here $\mu_1, \nu_1$ and $\alpha_1$ are defined in (\ref{eqn20}). The constant parameter $\mu_2, \nu_2$ and $\alpha_2$ are defined in \cite{1} by 
\begin{equation}\label{eqn6.2}
\mu_2 := \biggl(1+\frac{1}{\gamma}\biggl)^2, \qquad \nu_2 := \frac{(1+\gamma)^{4/3}}{\gamma}, \qquad \alpha_2 = \frac{2}{3}. 
\end{equation}

Our next step is analyze $K_{ji}$. Let's take $K_{11}$ as an example. As we did in (\ref{eqn11} -- \ref{eqn16}), consider the kernel $K_{11}(\eta, \zeta)$, with the change of variables
\[
\eta = \mu_1 -\frac{\nu_1}{M^{\alpha_1}}(x+u), \qquad\zeta = \mu_1 -\frac{\nu_1}{M^{\alpha_1}}(x+v)
\]
we know that $K_{11}: L^2(0, \xi_1)\to L^2(0, \xi_1)$ is equivalent to $\mathcal{K}_{11}: L^2(0, \mu_1M^{\alpha_1}/\nu_1 - x)\to L^2(0, \mu_1M^{\alpha_1}/\nu_1 - x)$ defined in (\ref{eqn14}). The analysis for the other $K_{ji}$'s are the same. Finally we can arrive at the following equation. 
\begin{equation}\label{eqn6.3}
\mathbb{P}(\widetilde{\lambda}_{\min}\leq x, \widetilde{\lambda}_{\max}\leq y) = \det\biggl(I-
\left(
\begin{array}{cc}
\mathcal{K}_{11} & W_M^{-1}\mathcal{K}_{12} \\
W_M\mathcal{K}_{21} & \mathcal{K}_{22} \\
\end{array}
\right)
\biggl)
\end{equation}
Here 
\[
\widetilde{\lambda}_{\min} := \frac{\gamma M^{2/3}}{(\gamma-1)^{4/3}}\Big[(1-\gamma^{-1})^2-\lambda_{\min}\Big], \qquad \widetilde{\lambda}_{\max} := \frac{\gamma M^{2/3}}{(\gamma+1)^{4/3}}\Big[\lambda_{\max}-(1+\gamma^{-1})^2\Big]
\]
is the proper scaling of $\lambda_{\min}, \lambda_{\max}$ and the $\mathcal{K}_{ji} (i, j\in\{1, 2\})$ is the integral operator from $\widetilde{J}_i$ to $\widetilde{J}_j$ where $\widetilde{J}_1 := (0, \mu_1M^{\alpha_1}/\nu_1 - x)$ and $\widetilde{J}_2 = (0, \infty)$. The kernels on the diagonal are 
\begin{eqnarray}
\mathcal{K}_{11}(u, v) & = & -\int_0^\infty \mathcal{H}_1(y+x+u)\mathcal{J}_1(y+x+v)dy \label{eqn6.4}\\ 
\mathcal{K}_{22}(u, v) & = & \int_0^\infty \mathcal{H}_2(y+x+u)\mathcal{J}_2(y+x+v)dy \label{eqn6.4.1}
\end{eqnarray}
with $\mathcal{H}_1(u), \mathcal{J}_1(v)$ defined by (\ref{eqn12}) and (\ref{eqn13}). The $\mathcal{H}_2(u), \mathcal{J}_2(v)$ appeared in (\ref{eqn6.4}) are defined in \cite{1} by 
\begin{eqnarray}
\mathcal{H}_2(u) := \frac{\nu_2M^{1/3}}{2\pi}\int_{\Gamma_2}e^{-\nu_2M^{1/3}u(z-q_2)}e^{-M\mu_2(z-q_2)}z^M\prod_{\ell=1}^N\frac{1}{z-\pi_\ell}dz \label{eqn6.5} \\
\mathcal{J}_2(u) := \frac{\nu_2M^{1/3}}{2\pi}\int_{\Sigma_2}e^{\nu_2M^{1/3}v(\omega-q_2)}e^{M\mu_2(\omega-q_2)}\omega^{-M}\prod_{\ell=1}^N({\omega-\pi_\ell})d\omega.\label{eqn6.6}
\end{eqnarray}
In Section \ref{sec4} we have already shown that $\mathcal{H}_1(u), \mathcal{J}_1(v)$ will tend to the desired limit at the speed of $M^{-1/3}$ as $M$ tends to infinity. In \cite{1}, Baik, Ben Arous and P\'{e}ch\'{e}  showed that $\mathcal{H}_2(u), \mathcal{J}_2(v)$ will also tend to the desired form at the speed of $M^{-1/3}$. That is to say, the diagonal part in (\ref{eqn6.3}) has the non-trivial limit because of the correct scaling in (\ref{eqn6.1}). In order to prove that $\lambda_{\min}$ and $\lambda_{\max}$ are asymptotically independent, our strategy is to show that the off-diagonal terms $W_M\mathcal{K}_{21}$ and $W_M^{-1}\mathcal{K}_{12}$ in (\ref{eqn6.3}), properly scaled, will tend to zero. Hence the matrix in (\ref{eqn6.3}) will tend to a diagonal matrix. Since the determinant of a diagonal matrix equals to the product of the determinant of its diagonal parts, the joint probability $\mathbb{P}(\widetilde{\lambda}_{\min}\leq x, \widetilde{\lambda}_{\max}\leq y)$ is approximately $\mathbb{P}(\widetilde{\lambda}_{\min}\leq x)\mathbb{P}(\widetilde{\lambda}_{\max}\leq y)$. Now we carry on our strategy in details. 

First consider the off-diagonal terms in (\ref{eqn6.3}). First define the constant 
\begin{equation}\label{const}
W_M = \frac{e^{Mf_1(p_1)}}{e^{Mf_2(p_2)}}\cdot \frac{\widehat{g}_1(p_1)}{\widehat{g}_2(p_2)}\cdot\frac{(\nu_1 M^{1/3})^{k_1}}{(\nu_2 M^{1/3})^{k_2}}.
\end{equation}
($p_1, p_2$ and the functions $f_1(\cdot), f_w(\cdot), \widehat{g}_1(\cdot), \widehat{g}_2(\cdot)$ will be defined below).

Using the change of variables as we did in (\ref{eqn11} -- \ref{eqn16}), we get  
\begin{eqnarray}
W_M^{-1}\mathcal{K}_{12}(u, v)  & = & -\frac{1}{\nu_1 M^{1/3}}\int_{\Gamma_1}dz\int_{\Sigma_2}d\omega \widehat{H}_1(x+u; z)\widehat{J}_2(y+v; \omega)\cdot\frac{1}{\omega-z}, \label{eqn6.10}\\
W_M\mathcal{K}_{21}(u, v)  & = & -\frac{1}{\nu_2 M^{1/3}}\int_{\Gamma_2}dz\int_{\Sigma_1}d\omega \widehat{H}_2(y+u; z)\widehat{J}_1(x+v; \omega)\cdot\frac{1}{\omega-z}. \label{eqn6.9}
\end{eqnarray}
Here 
\begin{eqnarray}
\widehat{H}_1(u; z) & = & e^{M(f_1(z)-f_1(p_1))}e^{\nu_1M^{1/3}u(z-q_1)}\cdot \frac{\widehat{g}_1(z)}{\widehat{g}_1(p_1)}\cdot\frac{1}{[\nu_1M^{1/3}(z-p_1)]^{k_1}}, \label{eqn6.11}\\
\widehat{H}_2(u; z) & = & e^{M(f_2(z)-f_2(p_2))}e^{-\nu_2M^{1/3}u(z-q_2)}\cdot \frac{\widehat{g}_2(z)}{\widehat{g}_2(p_2)}\cdot\frac{1}{[\nu_2M^{1/3}(z-p_2)]^{k_2}}, \label{eqn6.12} \\
\widehat{J}_1(v; \omega) & = & e^{-M(f_1(\omega)-f_1(p_1))}e^{-\nu_1M^{1/3}u(\omega-q_1)}\cdot \frac{\widehat{g}_1(p_1)}{\widehat{g}_1(\omega)}\cdot{[\nu_1M^{1/3}(\omega-p_1)]^{k_1}}, \label{eqn6.13}\\
\widehat{J}_2(v; \omega) & = & e^{-M(f_2(\omega)-f_2(p_2))}e^{\nu_2M^{1/3}u(\omega-q_2)}\cdot \frac{\widehat{g}_2(p_2)}{\widehat{g}_2(\omega)}\cdot{[\nu_2M^{1/3}(\omega-p_2)]^{k_2}}. \label{eqn6.14} 
\end{eqnarray}
The function $f_1(z), f_2(z), \widehat{g}_1(z), \widehat{g}_2(z)$ are defined by 
\begin{eqnarray}
f_i(z) & = & -\mu_i(z-q_i)+\log z -\frac{1}{\gamma^2}\log(z-1), \qquad i\in\{1, 2\}.\label{eqn6.15} \\
g_1(z) &:=& (z-1)^{r_1}\prod_{\ell=N-r_1+1}^{N-k_1}\frac{1}{z-\pi_\ell}, \quad  g_2(z) := (z-1)^{r_2}\prod_{\ell=k_2+1}^{r_2}\frac{1}{z-\pi_\ell}
\end{eqnarray}
and the parameter $p_1, p_2, q_1, q_2$ are defined by 
\begin{equation}\label{eqn6.17}
p_1 = \frac{\gamma}{\gamma-1}, \quad p_2 = \frac{\gamma}{\gamma+1}, \qquad q_1 =p_1+\frac{\epsilon}{\nu_1M^{1/3}}, \quad q_2 =p_2-\frac{\epsilon}{\nu_2M^{1/3}}.
\end{equation}

From some simple calculation, we know that $p_1, p_2$ are the saddle point of $f_1(z)$ and $f_2(z)$, respectively. Below we will perform our saddle point analysis on $W_M^{-1}\mathcal{K}_{12}$ and $W_M\mathcal{K}_{21}$. Let's start from $W_M^{-1}\mathcal{K}_{12}(u, v)$ first. In Figure \ref{fig5} we plot the contour $\Gamma_1$ and $\Sigma_2$ in (\ref{eqn6.10}). The contour $\Gamma_1$ is defined in Figure \ref{fig2}, Section \ref{sec4}. Moreover the contour $\Sigma_2$ is defined in \cite{1}. We quote their definition below. 
\begin{figure}[!h]
    \centering
    \includegraphics[scale=.7]{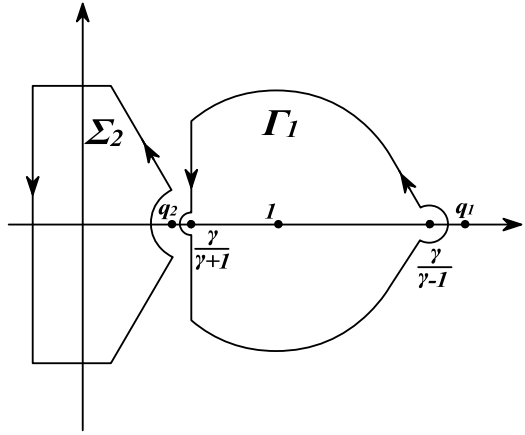}
    \caption{The Contour $\Gamma_1$ and $\Sigma_2$.}
    \label{fig5}
\end{figure}

Define $\Sigma_2 := \bigcup_{k=0}^3\Sigma_{2, k}\cup\overline{\bigcup_{k=0}^3\Sigma_{2, k}}$ where 
\begin{eqnarray*}
\Sigma_{2, 0} & := & \{p_2 + 3\epsilon(\nu_2 M^{1/3})e^{i(\pi-\theta)}: 0\leq\theta\leq \pi/3\}, \\
\Sigma_{2, 1} & := & \{p_2+te^{2i\pi/3}: 3\epsilon(\nu_2 M^{1/3})\leq t\leq 2p_2\}, \\
\Sigma_{2, 2} & := & \{p_2 + 2p_2e^{2i\pi/3}-x: 0\leq x\leq R\}, \\
\Sigma_{2, 3} & := & \{-R+i(\sqrt{3}p_2-y): 0\leq y\leq \sqrt{3}p_2\}
\end{eqnarray*}
where $R$ is a pre-fixed constant which is sufficiently large and independent of $M$. 

The intuition of the proof is as follows. The saddle point of $f_1(z), f_2(z)$ is $p_1, p_2$ respectively. Moreover in our contour we leave the saddle points along the steepest decent direction, in (\ref{eqn6.10}) the integral on $z$ will be mostly concentrated on the point $p_2$ and the integral on $\omega$ will be mostly concentrated on the point $p_1$. But by our choice of the constant $W_{M}$, $\widehat{H}_1(x+u; z)$ and $\widehat{J}_2(y+v; \omega)$ will be bounded at these two saddle points. Moreover around these two points $(\omega-z)^{-1}\approx(p_2-p_1)^{-1}$ will also be bounded. Hence due to the factor $M^{-1/3}$ in front of the integral of (\ref{eqn6.10}), $W_M^{-1}\mathcal{K}_{12}(u, v)$ will tend to zero, at the speed of $M^{-1/3}$.  To paraphrase our intuition into a rigorous proof, we will arrive at the following lemma. 

%
%
\begin{lemma}\label{lemma6.1}
These exist some constant $C>0, c>0, M_0>0$ such that uniformly for $u>U, v>V$, we have
\begin{eqnarray}\label{eqn6.20}
|W_M^{-1}\mathcal{K}_{12}(u, v)|& \leq& C(e^{-cu-cv})M^{-1/3}, \label{eqn6.20}\\
|W_M\mathcal{K}_{21}(u, v)| &\leq& C(e^{-cu-cv})M^{-1/3}.  \label{eqn6.20.1}
\end{eqnarray}
\end{lemma}
%
%
\begin{proof}
Here we just prove (\ref{eqn6.20}) as the proof for (\ref{eqn6.20.1}) will be similar. 
Let's define $\delta < \min\{p_1/10, p_2/10, (p_1-p_2)/10\}$ to a sufficiently small constant independent of $M$. Just like what we did in Section \ref{sec4} and \ref{sec5}, we decompose the contour by $\Gamma_1 := \Gamma_1' + \Gamma_1''$, $\Sigma_2 := \Sigma_2' + \Sigma_2''$ by 
\begin{eqnarray*}
\Gamma_1' = \{z\in\Gamma_1 : |z-p_1|\leq\delta\}, \qquad \Gamma_1'' = \Gamma_1\backslash\Gamma_1', \\
\Sigma_2' = \{z\in\Sigma_2 : |z-p_2|\leq\delta\}, \qquad \Sigma_2'' = \Sigma_2\backslash\Sigma_2'. 
\end{eqnarray*}
We first analyze $\Gamma_1$. By Lemma \ref{lemma1} we know that $\Re(f(z))$ is decreasing on $\Gamma_1'$. Hence there exist some constant $c$ such that 
\begin{equation}\label{eqn6.21}
\forall z\in\Gamma_1', \Re\Big[f(z)-f(p_1)\Big]\leq 0, \qquad \forall z\in\Gamma_1'', \Re\Big[f(z)-f(p_1)\Big]\leq -c.
\end{equation}
For $\Sigma_2$ we can also deduce the same result. From Lemma 3.2 in \cite{1}, we have
\begin{equation}\label{eqn6.22}
\forall \omega\in\Sigma_2', \Re\Big[f(z)-f(p_1)\Big]\geq 0, \qquad \forall \omega\in\Sigma_2'', \Re\Big[f(\omega)-f(p_2)\Big]\geq c.
\end{equation}
For $z\in \Gamma_1'$, we note three results proved before. (1) $\exp[\nu_1M^{1/3}u\Re(z-q_1)]\leq \exp(-\epsilon u/2)$ by the definition of the contour; (2) $|\widehat{g}_1(z)/\widehat{g}_1(p_1)|$ is bounded; and (3) $|\nu_1M^{1/3}(z-p_1)|^{-k_1}\leq \epsilon^{-k_1}$ is bounded.  These three results, together with (\ref{eqn6.21}) give that 
\begin{equation}\label{eqn6.23}
\int_{\Gamma_1'}|\widehat{H}_1(u; z)|dz \leq Ce^{-cu}
\end{equation}
for some constant $C, c>0$. 
Now we proceed to prove
\begin{equation}\label{eqn6.24}
\int_{\Sigma_2'}|\widehat{J}_2(v; \omega)|d\omega \leq Ce^{-cv}
\end{equation}
Similar to (\ref{eqn37} -- \ref{eqn41}), we can choose $\delta$ small enough such that for $|\omega-p_2|<\delta$,
\begin{equation}\label{eqn6.25}
\Re\Big[f(\omega)-f(p)\Big]\geq \frac{\nu_2^3}{6}|\omega-p|^3. 
\end{equation}
Actually in \cite{1} the authors proved that (\ref{eqn6.25}) holds as long as 
\[
\delta < \min\biggl\{\frac{1}{2(1+\gamma)}, \frac{\nu_2^3}{6(4^3+4(1+\gamma)^4)}\biggl\}.
\]
Now for $\omega \in\Sigma_2'$, 
\begin{eqnarray}
|\widehat{J}_2(v, \omega)| &\leq & Ce^{-cv}e^{-(\nu_2M^{1/3}|\omega-p_2|)^3/6}\biggl|\frac{\widehat{g}_2(p_2)}{\widehat{g}_2(\omega)}\biggl|\Big[\nu_2M^{1/3}|\omega-p_2|\Big]^{k_2} \nonumber\\
& \leq & Ce^{-cv}
\label{eqn6.26}
\end{eqnarray}
since that $|{\widehat{g}_2(p_2)}/{\widehat{g}_2(\omega)}|$ is bounded and that $e^{-cx}x^{k_2}$ is bounded for $x\geq0$. From (\ref{eqn6.26}) we obtain (\ref{eqn6.24}) directly. Having analyzed the behavior of $\widehat{H}_1, \widehat{J}_2$ on $\Gamma_1'$ and $\Sigma_2'$, our next step is to prove that they decay exponentially fast on $\Gamma_1''$ and $\Sigma_2''$. Let's take $\widehat{H}_1(u; z)$ as an example. The proof of $\widehat{J}_2(v; \omega)$ will be similar. Indeed, by the second part of (\ref{eqn6.21}) and by the fact that $\widehat{g}_1(z)/\widehat{g}_1(p_1)$ is bounded, we obtain 
\begin{equation}\label{eqn6.27}
\int_{\Gamma_1''}|\widehat{H}_1(u; z)|dz \leq Ce^{-cu}e^{-cM}. 
\end{equation}
Similar results for $\widehat{J}_2(v; \omega)$ can also be obtain. 

As the last step in proving our lemma, we have 
\begin{equation}\label{eqn6.28}
W_M^{-1}\mathcal{K}_{12}(u, v) = -\frac{1}{\nu_1M^{1/3}}\Big[I_1+I_2+I_3+I_4\Big].
\end{equation}
Here 
\begin{equation}\label{eqn6.29}
I_1 := \int_{\Gamma_1'}dz\int_{\Sigma_2'}d\omega \widehat{H}_1(x+u; z)\widehat{J}_2(y+v; \omega)\cdot\frac{1}{\omega-z},
\end{equation}
and $I_2, I_3, I_4$ are just the integral of (\ref{eqn6.29}) on $(\Gamma_1'', \Sigma_2'), (\Gamma_1', \Sigma_2'')$ and $(\Gamma_1'', \Sigma_2'')$, respectively.  

In $I_1$, $z\in\Gamma_1'$ and $\omega\in\Sigma_2''$, we have $|\omega-z|^{-1}\leq((p_1-p_2)/2)^{-1}$ is bounded above. Hence by (\ref{eqn6.23}) and (\ref{eqn6.24}) we have 
\begin{equation}\label{6.30}
|I_1| \leq Ce^{-cu-cv}. 
\end{equation}
For $I_2$, $z\in\Gamma_1'', \omega\in\Sigma_2'$. We just use the trivial bound $|\omega-z|^{-1}\leq CM^{1/3}$ for some constant $C$, by the definition of our contour. Hence by (\ref{eqn6.27}) and (\ref{eqn6.26}) we get
\begin{equation}\label{eqn6.31}
|I_2|\leq Ce^{-cu-cv}e^{-cM}\cdot M^{1/3} \leq Ce^{-cu-cv}e^{-cM}. 
\end{equation}
With similar analysis, we can prove that (\ref{eqn6.31}) also holds for $I_3$ and $I_4$. Note that in (\ref{eqn6.28}) we have an extra $M^{-1/3}$ in front of the integral. This finishes the proof of (\ref{eqn6.20}) in our lemma. 
\end{proof}
%
%
As a simple consequence, Theorem \ref{thm2} holds. Here's the rest of the proof for the theorem. 
%
%
\begin{proof}[Proof of Theorem \ref{thm2}]
By Lemma (\ref{lemma6.1}) we have 
\begin{equation}\label{eqn6.33}
\|W_M^{-1}\mathcal{K}_{21}\|_1 \leq CM^{-1/3}, \qquad \|W_M\mathcal{K}_{12}\|_1\leq CM^{-1/3}
\end{equation}
for some constant $C$. Here $\|\cdot\|_1$ is the trace norm of the operator. Since the Fredholm determinant is a locally Lipchitz continuous function with respect to the trace norm, from (\ref{eqn6.3}) we obtain
\begin{eqnarray*}
\mathbb{P}(\widetilde{\lambda}_{\min}\leq x, \widetilde{\lambda}_{\max}\leq y) & = & \det\biggl(I-
\left(
\begin{array}{cc}
\mathcal{K}_{11} & W_M^{-1}\mathcal{K}_{12} \\
W_M\mathcal{K}_{21} & \mathcal{K}_{22} \\
\end{array}
\right)
\biggl) \\
& = & \det\biggl(I-
\left(
\begin{array}{cc}
\mathcal{K}_{11} & 0 \\
0 & \mathcal{K}_{22} \\
\end{array}
\right)
\biggl) + \mathcal{O}(M^{-1/3}) \\
& = & \det(I-\mathcal{K}_{11})\det(I-\mathcal{K}_{22}) + \mathcal{O}(M^{-1/3}) \\
&\to& F_{k_1}(x)F_{k_2}(y). 
\end{eqnarray*}
\end{proof}
%
%

\section{Conclusion}\label{sec7}
In this paper, we studied the spiked population model to establish two results: (1) the asymptotic distribution of the smallest eigenvalue of the sample covariance matrix, and (2) that the largest and the smallest eigenvalues are independent. Our approach is based on the convergence of operators under the trace norm. 

For the smallest eigenvalue, when the spike is weak (all above the threshold $1-\gamma^{-1}$), the local fluctuation is of order $\mathcal{O}(M^{-2/3})$, and will converge to the generalized Tracy-Widom law under the correct scaling. When the spike is strong enough to pull the smallest eigenvalue out of the Mar\v{c}enko-Pastur sea, it wil fluctuate at the order of $\mathcal{O}(M^{-1/2})$, following the generalized Gaussian distribution. We also proved that the largest and smallest eigenvalues are independent. Combined with the asymptotic behavior of $\lambda_{\max}$ proposed in \cite{1}, we can establish the joint distribution of $(\lambda_{\min}, \lambda_{\max})$.

%
%

\end{document}